\documentclass[a4paper]{amsart}
\usepackage{amsmath,amsthm,amsfonts,amssymb,hyperref,color,comment}
\usepackage{graphicx}

\usepackage{float}
\numberwithin{equation}{section}
\theoremstyle{plain}

\newtheorem{theorem}{\sc Theorem}[section]

\newtheorem{corollary}[theorem]{\sc Corollary}
\newtheorem{definition}[theorem]{\sc Definition}

\newtheorem{lemma}[theorem]{\sc Lemma}

\theoremstyle{remark}
\newtheorem{remark}[theorem]{\sc Remark}

\newcommand{\be}{\begin{equation}}
\newcommand{\ee}{\end{equation}}

\newcommand{\N}{\mathbb{N}}
\newcommand{\n}{^{(n)}}
\newcommand{\PP}{\mathbb{P}}
\newcommand{\prt}[1]{\frac{d }{d t}  {#1}(t)}

\begin{document}

\title[Fluid Limit for the Poisson Encounter-Mating Model]{Fluid Limit for the Poisson Encounter-Mating Model}

\author{Onur G\"un}
\address{Onur G\"un\\ Weierstrass Institute\\ Mohrenstrasse 39\\ 10117 Berlin\\ Germany.}
\email{Onur.Guen@wias-berlin.de}

\author{Atilla Yilmaz}
\address{Atilla Yilmaz, Department of Mathematics, Ko\c{c} University, Sar\i yer, Istanbul 34450, Turkey.}
\email{atillayilmaz@ku.edu.tr}

\date{Revised on November 18, 2016.}

\subjclass[2010]{92D25, 60J28, 60F15.} 
\keywords{Density dependent population process, pair formation, assortative mating, panmixia, mating preferences, mating pattern, fluid approximation, Lotka-Volterra equations, replicator equations.}

\begin{abstract}

Stochastic encounter-mating (SEM) models describe monogamous permanent pair formation in finite zoological populations of multitype females and males. In this article, we study SEM with Poisson firing times. First, we prove that the model enjoys a fluid limit as the population size diverges, i.e., the stochastic dynamics converges to a deterministic system governed by coupled ODEs. Then, we convert these ODEs to the well-known Lotka-Volterra and replicator equations from population dynamics. Next, under the so-called fine balance condition which characterizes panmixia, we solve the corresponding replicator equations and give an exact expression for the fluid limit. Finally, we consider the case with two types of females and males. Without the fine balance assumption, but under certain symmetry conditions, we give an explicit formula for the limiting mating pattern, and then use it to characterize assortative mating. 

\end{abstract}

\maketitle

\section{Introduction}\label{sec1}

\subsection{The model}
Consider a zoological population consisting of $n$ females and $n$ males, divided into $k$ types which are labeled $1,\dots,k$. We denote by $x_i\n\geq 0$ the number of type-$i$ females and by $y_j\n\geq 0$ the number of type-$j$ males, for $i,j\in[k]:=\{1,\dots,k\}$. 
To each type-$i$ female (resp.\ type-$j$ male) a Poisson process with rate $\alpha_i$ (resp.\ $\beta_j$) is attached. These Poisson processes are mutually independent and they give the so-called firing times of the animals.
The mating preferences of the animals depend on their types and form a $k\times k$ matrix $P=(p_{ij})_{i,j\in[k]}$, with $0<p_{ij}\leq 1$. Under these assumptions, the dynamics of the population is as follows. Initially all individuals are single. 
At any time, when the Poisson clock of one of the single individuals rings (by the Poisson assumption no two individuals' clocks ring at the same time), it chooses a single individual from the opposite sex, uniformly at random, to form a temporary pair. Next, if this temporary pair is comprised of a type-$i$ female and a type-$j$ male, it becomes a permanent pair with probability $p_{ij}$ and the individuals in that pair leave the singles pool; otherwise the temporary pair is broken and the individuals go back to the singles pool.  We refer to this two-stage permanent pair formation model as Poisson encounter-mating (Poisson EM).  Observe that the number of types present in the female and male populations need not be the same. Indeed, setting for example $x_i\n=0$ would take type-$i$ females out of the picture. 

We designate by $Q_{ij}\n(t),\;t\geq 0,$ the number of (permanent) type-$ij$ pairs at time $t$. Here, the first index always refers to the type of the female and the second to the type of the male. We call the $k\times k$ matrix-valued process $Q\n(t)=(Q\n_{ij}(t))_{i,j\in[k]}$ the pair-type process. Since the Poisson processes are memoryless, $Q\n$ is a pure jump continuous-time Markov process. In order to formally define $Q\n$, we briefly introduce some notation. Let $\mathcal{M}^{k\times k}(A)$ denote the set of $k\times k$ matrices whose entries are in $A\subseteq \mathbb{R}$. For $M=(M_{ij})_{i,j\in[k]}\in\mathcal{M}^{k\times k}(A)$ we define the $i$-th row sum, the $j$-th column sum and the grand total of $M$, respectively, as
\begin{equation*}
M_{i,\cdot}=\sum_{j'=1}^k M_{ij'},\quad M_{\cdot,j}=\sum_{i'=1}^k M_{i'j},\quad M_{tot}=\sum_{i'=1}^k\sum_{j'=1}^k M_{i'j'}. 
\end{equation*}
We denote by $I^{ij}$ the $k\times k$ matrix whose entries are zero except the $ij$-th entry, which is one. Throughout this article we use the max norm on $\mathcal{M}^{k\times k}(A)$ given by $|M|=\max_{i,j\in[k]}|M_{ij}|$. Since all matrix norms are equivalent, our results are valid for any choice of norm.

The pair-type process $Q\n$ is a continuous-time Markov process taking values in $\mathcal{M}^{k\times k}(\N\cup\{0\})$ that has jumps of size 1, more precisely, the transitions are from $M$ to $M+I^{ij}$ for $i,j\in[k]$. The transition rates are given by
\begin{equation}\label{trates}
\rho^{(n)}(M,M+I^{ij})=\frac{\pi_{ij}\big(x_i\n-M_{i,\cdot}\big)\big(y_j\n-M_{\cdot,j}\big)}{n-M_{tot}}
\end{equation}
where
\begin{equation*}
\Pi=(\pi_{ij})_{i,j\in[k]},\quad \pi_{ij}=p_{ij}(\alpha_i+\beta_j),
\end{equation*}
with the convention that $\rho^{(n)}(M,\cdot)\equiv 0$ for $M$ with $M_{tot}=n$. 

Let us explain the formula in (\ref{trates}). When the pair-type formation at a time is $M$, the number of type-$i$ females (resp. type-$j$ males) in the singles pool is $x_i\n-M_{i,\cdot}$ (resp. $y_j\n-M_{\cdot,j}$). Also, by the description of the model, the total number of single females is always equal to that of single males and given by $n-M_{tot}$. A new type-$ij$ pair is formed in two ways: either
the clock of a type-$i$ single female rings, this female encounters a type-$j$ single male to form a temporary pair, and finally, this pair becomes permanent; or similar has to happen with a type-$j$ single male's clock ringing. In the first scenario, the total rate with which the clock of a type-$i$ single female rings is $\alpha_i(x_i\n-M_{i,\cdot})$, the probability that it samples a type-$j$ male from single males is $(y_j\n-M_{\cdot,j})/(n-M_{tot})$, the probability that the temporary pair formed becomes permanent is $p_{ij}$, and the product of these terms gives the rate of this event. The corresponding terms in the second scenario are $\beta_j(y_j\n-M_{\cdot,j})$, $(x_i\n-M_{i,\cdot})/(n-M_{tot})$ and $p_{ij}$. Finally, the sum of the rates of these two events gives (\ref{trates}).

Since $Q\n$ is a pure jump Markov process for every $n$, it is possible to define the whole family $\{Q\n : n\in\N\}$ via a collection of independent standard Poisson processes whose joint distribution we denote by $\PP$ (see Section \ref{subsecFL}). We are interested in the infinite population asymptotics of the model, therefore we assume that there are non-negative numbers $x_1,\dots,x_k$ and $y_1,\dots,y_k$ such that for all $i,j\in[k]$, as $n\to\infty$
\be\label{convxy}
n^{-1}x_i\n\longrightarrow x_i,\quad n^{-1}y_j\n\longrightarrow y_j.
\ee
Note that 
$
x_1+\cdots+x_k=y_1+\cdots+y_k=1.
$ We refer to such a collection of numbers $x_1,\dots,x_k,y_1,\dots,y_k$ as an infinite population from the species.

The pair-type process $Q\n$ naturally stops at
\begin{equation*}
T_n:=\inf \{t\geq 0:\;Q_{tot}\n(t)=n\},
\end{equation*}
that is, when the singles pool is depleted and every individual is in a permanent pair.
$Q\n(T_n)$ is called the mating pattern of the population and is of central importance in this paper. Note that $Q\n(T_n)$ is a random $k\times k$ matrix (or contingency table) whose $i$-th row sum is $x_i\n$ and $j$-th column sum is $y_j\n$ for all $i,j\in[k]$.
We always assume that $p_{ij}>0$ and $\alpha_i+\beta_j>0$ for all $i,j\in[k]$. Hence, almost surely $T_n<\infty$.

\subsection{Panmixia, homogamy and heterogamy}

One fundamental question about the mating pattern is whether correlations exist between female and male types. Zero correlations correspond to the case where the relative frequency of type-$ij$ pairs is given by the product of the relative frequencies of type-$i$ females and type-$j$ males, which has been called ``panmixia" in the literature. Since we investigate Poisson EM as the population size diverges and establish a strong limit theorem for the mating pattern $Q\n(T_n)$, we naturally use the following definition of panmixia.

\begin{samepage}
\begin{definition}
An infinite population $x_1,\dots,x_k,y_1,\dots,y_k$ is said to be panmictic if $\PP$-a.s.
\begin{equation*}
\lim_{n\to\infty}n^{-1}Q_{ij}\n(T_n) = x_i y_j,\quad\quad\quad   \forall i,j\in[k].
\end{equation*}
The species is said to be panmictic if every infinite population from the species is panmictic.
\end{definition}
\end{samepage}

Complementing the concept of panmixia is assortative mating. Homogamy (resp.\ heterogamy) describes the situations where there are positive (resp.\ negative) correlations in the mating pattern between females and males with similar types. In order to make the definition of assortative mating precise, one needs a (genotypical or phenotypical) distance on the set of types. Such a structure for types must be reflected on preferences and this requires a more complex model.  However, when $k=2$, we can conveniently define assortative mating since there is a unique metric on $\{1,2\}$.  Moreover, in this case, there is homogamy (resp.\ heterogamy) for type-1  if and only if there is homogamy (resp.\ heterogamy) for type-2. These observations lead to the following definition.
\begin{definition}
For $k=2$, an infinite population $x_1,x_2,y_1,y_2$ with $x_1x_2y_1y_2\not=0$ is said to be homogamous if $\PP$-a.s.
\begin{equation*}
\lim_{n\to\infty} n^{-1}Q_{12}\n(T_n) < x_1 y_2,
\end{equation*}
and heterogamous if $\PP$-a.s.
\begin{equation*}
\lim_{n\to\infty} n^{-1}Q_{12}\n(T_n) > x_1 y_2.
\end{equation*}
The species is said to be homogamous (resp.\ heterogamous) if every such infinite population from the species is homogamous (resp.\ heterogamous).
\end{definition} 

Note that definitions of both panmixia and homogamy/heterogamy assume the existence of the infinite population limit of the normalized mating pattern and that this limit is the same for all sequences of finite populations satisfying (\ref{convxy}), which are shown in Section \ref{sec2}.  For the corresponding definitions in the context of finite populations, one has to replace limits with expectations (see \eqref{ortakavm} for $k=2$). Also, observe that in the definition of homogamy/heterogamy we exclude the cases where one type is absent, since otherwise the system is trivial and there is panmixia for all choices of parameters. 

\subsection{Previous results}\label{subsecprev}

In \cite{Gim88a}, Gimelfarb introduced two discrete-time models for permanent monogamous pair formation: individual and mass encounter-mating. In the first model, at each time step, one single female and one single male are selected, both uniformly at random, to form a temporary pair and this pair becomes permanent exactly as in the Poisson EM model with probability $p_{ij}$. Observe that if we set, say,  $\alpha_i=0$ and $\beta_j=1$ for all $i,j\in[k]$, then the dynamics of Gimelfarb's individual encounter-mating model is the same as the embedded discrete-time chain of the pair-type process $Q\n$ of Poisson EM, and in particular, the mating patterns of the two models coincide. The mass encounter-mating model has a very different encounter mechanism where, at each time step, all the single females and males form temporary pairs according to a permutation chosen uniformly at random, while the mechanism of permanent pair formation from temporary pairs is as before. The main conceptual conclusion of Gimelfarb was that the mating pattern depends not only on the preferences but also on the encounter mechanism. Moreover, given the encounter mechanism, different mating preferences can lead to the same mating pattern. He then stated conditions on the parameters of the models that he conjectured to be sufficient for panmixia, supported the one for mass encounter with a non-rigorous argument, and provided only numerical evidence in the individual encounter case.

In \cite{GunYil14a}, we introduced the stochastic encounter-mating (SEM) model to unify and generalize Gimelfarb's models. The key feature of this generalization is the introduction of firing times which allows one to define a wide range of models and take advantage of their invariance under certain changes of parameters. We investigated in detail the special case where $p_{ij}=1$ for all $i,j\in[k]$, that is, definite mating upon encounter, and proved among other things that there is panmixia for all firing time distributions and that the firing times and the mating pattern are independent.  As we have already seen, the pair-type process of Poisson EM is a continuous-time Markov process whose rates depend on the parameters of the model through $\pi_{ij}=p_{ij}(\alpha_i+\beta_j)$. Hence, one can play with the parameters without changing the model as long as $\pi_{ij}$'s stay the same. Using this and our analysis of the case with definite mating upon encounter, we concluded that the model exhibits panmixia if there are non-negative numbers $\bar \alpha_i$ and $\bar \beta_j$ such that $\pi_{ij}=p_{ij}(\alpha_i+\beta_j)=1(\bar \alpha_i+\bar \beta_j)$ for every $i,j\in[k]$. We record this condition for future reference.
\begin{definition}\label{defnfineb} We say that Poisson EM satisfies the fine balance condition if there exist non-negative numbers $\bar{\alpha}_1,\dots,\bar{\alpha}_k$ and $\bar \beta_1,\dots,\bar{\beta}_k$ such that
\begin{equation}\label{fb}
\pi_{ij}=\bar \alpha_i+\bar \beta_j,\quad\quad\quad \forall i,j\in[k].
\end{equation}
Equivalently,
\begin{equation*}
\pi_{ij}+\pi_{i'j'}=\pi_{ij'}+\pi_{i'j},\quad \forall i,i',j,j'\in[k].
\end{equation*}
\end{definition}
The fine balance condition is precisely what Gimelfarb had conjectured in \cite{Gim88a} to be sufficient for panmixia in the context of individual encounter-mating. In \cite{GunYil14a}, we not only settled this conjecture, but also used a recursive argument to prove that the fine balance condition is necessary for the species to be panmictic. Moreover, under the fine balance condition we gave the distributions of the pair-type process $Q\n(t)$ and the mating pattern $Q\n(T_n)$. Finally, we answered the assortative mating question when $k=2$: for any $x_1\n x_2\n y_1\n y_2\n\ne0$,
\begin{align}
\label{ortakavm}
\begin{split}
\text{(homogamy for finite pop.)}\quad \mathbb{E}[Q_{12}\n(T_n)] &< n^{-1}x_1\n y_2\n\quad\text{if}\quad\pi_{11} + \pi_{22} > \pi_{12} + \pi_{21},\\
\text{(panmixia for finite pop.)}\quad \mathbb{E}[Q_{12}\n(T_n)] &= n^{-1}x_1\n y_2\n\quad\text{if}\quad\pi_{11} + \pi_{22} = \pi_{12} + \pi_{21},\\
\text{(heterogamy for finite pop.)}\quad \mathbb{E}[Q_{12}\n(T_n)] &> n^{-1}x_1\n y_2\n\quad\text{if}\quad\pi_{11} + \pi_{22} < \pi_{12} + \pi_{21}.
\end{split}
\end{align}
Here, $\mathbb{E}$ denotes expectation with respect to $\mathbb{P}$.


%
%

\subsection{Overview of results}\label{subsecoverview}

In this article, we analyze the dynamics of the Poisson EM model as the population size $n$ diverges. In Section \ref{sec2}, we start our investigation by observing that the pair-type process $Q\n$ is approximately a density dependent population process. Then, we show that $Q\n$ rescaled by $n$ converges $\PP$-a.s.\ in the sup norm up to any finite time, where the limiting (deterministic) process $Q(t)$ solves a system of coupled ODEs. More precisely, in Theorem \ref{thmLLN} we prove that, $\PP$-a.s.
\begin{equation*}
\lim_{n\to\infty} \sup_{0\leq t\leq T}\left|n^{-1}Q\n(t) - Q(t)\right|=0
\end{equation*}
for every $T\in[0,\infty)$, where $Q(t)=(Q_{ij}(t))_{i,j\in[k]}$ satisfies
\be\label{eqforq}
\prt{Q_{ij}} = \frac{\pi_{ij}\big(x_i-Q_{i,\cdot}(t)\big)\big(y_j-Q_{\cdot,j}(t)\big)}{1-Q_{tot}(t)},\quad \text{with}\quad Q_{ij}(0)=0.
\ee
This type of generalization of the law of large numbers (LLN), regarding the convergence of the rescaled paths of a pure jump Markov process to a solution of a system of ODEs, is known as the fluid limit and is due to \cite{K70}. Here, $Q$ represents the infinite population pair-type process and we use the terms pairs, singles, etc.\ for $Q$ as well. As a consequence of the fluid limit, we prove in Theorem \ref{cormat} that $\PP$-a.s.\ the mating pattern of the infinite population satisfies
$$\lim_{n\to\infty}n^{-1}Q\n(T_n) = Q(\infty) := \lim_{t\to\infty}Q(t).$$

After establishing these limit theorems, we focus on the evolution of $Q$. In Section \ref{sec3}, we relate the system of ODEs that describe $Q$ to the well-known Lotka-Volterra and replicator equations from population dynamics. Let $X_i(t)$, $Y_j(t)$ and $Z(t)$ denote the density of type-$i$ single females, type-$j$ single males and all single females (or males):
\be\label{defnxyz}
X_i(t):=x_i-Q_{i,\cdot}(t),\quad Y_j(t):=y_j-Q_{\cdot,j}(t),\quad Z(t):=1-Q_{tot}(t).
\ee
Then, for all $i,j\in[k]$,
\be\label{eqnxy}
\prt{X_i}=-\frac{X_i(t)}{Z(t)}\sum_{j=1}^k \pi_{ij} Y_j(t),\quad \prt{Y_j}=-\frac{Y_j(t)}{Z(t)}\sum_{i=1}^k \pi_{ij} X_i(t),
\ee
with $X_i(0)=x_i$ and $Y_j(0)=y_j$. Hence, up to a time change due to the $Z(t)$ term, this is a system of $2k$ Lotka-Volterra equations where the intrinsic growth (or decay) rate is 0 for all types and sexes. See Theorem \ref{ltrep} for the precise statement.\,Another important equation in population dynamics is the replicator equation, first introduced in \cite{TJ78}. Replicator equations describe the evolution of different types in a population under density dependent fitness functions and are often used in the context of evolutionary game theory. In general, a Lotka-Volterra equation with $l$ variables is equivalent to a replicator equation with $l+1$ variables, see \cite[Theorem 7.5.1]{HS98}. However, when intrinsic growth rates are constant, one does not need to increase the dimension to obtain a replicator equation. Indeed, the relative frequencies of types in the Lotka-Volterra system, up to a time change, solve the replicator equation with the same interactions. In particular, setting $A_i(t):=X_i(t)/Z(t)$ and $B_j(t):=Y_j(t)/Z(t)$ for all $i,j\in[k]$, we also prove in Theorem \ref{ltrep} that
\begin{equation}\label{repin}
\prt{A_i}=-A_i(t)\left[\sum_{j=1}^k \pi_{ij} B_j(t)-\bar C(t)\right],\quad \prt{B_j}=-B_j(t)\left[\sum_{i=1}^k \pi_{ij} A_i(t)-\bar C(t)\right],\quad
\end{equation}
where
\begin{equation*}
\bar C(t):=\sum_{i=1}^k\sum_{j=1}^k\pi_{ij}A_i(t)B_j(t).
\end{equation*}
We use (\ref{defnxyz})-(\ref{repin}) to deduce that
\begin{equation}\label{eqntotmass}
\prt{Z}=-{Z(t)}\sum_{i=1}^k\sum_{j=1}^k A_i(t) B_j(t).
\end{equation}
By (\ref{eqforq}), we observe that
\be\label{qode}
\prt{Q_{ij}}=\pi_{ij} Z(t)A_i(t) B_j(t),
\ee
and thus find a three-step procedure for obtaining a formula for $Q(t)$: (i) solve the replicator equations (\ref{repin}) for $A_i$'s and $B_j$'s; (ii) solve (\ref{eqntotmass}) to find the total mass $Z(t)$ of the corresponding (time-changed) Lotka-Volterra equations; and finally (iii) solve (\ref{qode}). 

In Section \ref{sec32}, we focus on the fine balance case. We carry out the three-step procedure and obtain a formula for $Q(t)$ for all $t\in[0,\infty]$, and in particular for the mating pattern $Q(\infty)$. Namely, in Theorem \ref{thmfb} we show that
\begin{equation*}
\begin{aligned}
&A_i(t)=\frac{x_i e^{-\bar\alpha_i t}}{\sum_{i'} x_{i'} e^{-\bar\alpha_{i'} t}},\quad B_{j}(t)=\frac{y_j e^{-\bar\beta_j t}}{\sum_{j'} y_{j'}e^{-\bar\beta_{j'} t}},\\&
Q_{ij}(t)=x_i y_j (1-e^{-\pi_{ij} t}),\quad\text{and}\quad Q_{ij}(\infty)=x_i y_j.
\end{aligned}
\end{equation*}
Here, recall that $\bar\alpha_i$ and $\bar\beta_j$ are from the fine balance condition given in Definition \ref{defnfineb}.
These formulas are fully consistent with those obtained in \cite[Theorem 3.6]{GunYil14a} for the expectations of the pair-type process and the mating pattern in the finite population setting, but here we employ a totally different approach via the replicator equations.

Finally, in Section \ref{sec4} we study the case $k=2$ with $\pi_{12}=\pi_{21}$ and $x_1=y_1$. Due to these symmetries, the evolution of the system can be reduced to that of only, say, females. As a result, the corresponding replicator dynamics is one-dimensional. More precisely, $A_i(t)=B_i(t)$ for all $t\geq 0$ and $i=1,2$, and setting $A_2(t)=1-A_1(t)$, we get
\be\label{foraintro}
\prt{A_1}=- (\pi_{11}+\pi_{22}-2\pi_{12})A_1(t)\big(1-A_1(t)\big)\big(A_1(t)-\gamma\big),
\ee
where 
\be\label{defgam}
\gamma=\frac{\pi_{22}-\pi_{12}}{\pi_{11}+\pi_{22}-2\pi_{12}}.
\ee
Note that in Section \ref{sec32} we explicitly solve the fine balance case which corresponds to $\pi_{11}+\pi_{22}-2\pi_{12}=0$, so we can exclude it, and (\ref{defgam}) is then well-defined. We derive a formula for $Q_{12}(t)$ in terms of $A_1(t)$ which depends on the value of $\gamma$:\\
For $\gamma=1$,
\begin{equation*}
Q_{12}(t)= \frac{\theta_1}{1-x_1}\int_{x_1}^{A_1(t)}\left(\frac{1 - x}{1 - x_1}\right)^{\theta_1 - 1}\left(\frac{x}{x_1}\right)^{-\theta_1-1}\exp\left\{{-}\theta_1\left(\frac1{1-x} - \frac1{1-x_1}\right)\right\}dx;
\end{equation*}
for $\gamma=0$,
\begin{equation*}
Q_{12}(t)= \frac{\theta_2}{x_1}\int_{1-x_1}^{1-A_1(t)}\left(\frac{1 - x}{x_1}\right)^{\theta_2 - 1}\left(\frac{x}{1-x_1}\right)^{-\theta_2-1}\exp\left\{{-}\theta_2\left(\frac1{1-x} - \frac1{x_1}\right)\right\}dx;
\end{equation*}
and for $\gamma\notin\{0,1\}$,
\begin{equation*}
Q_{12}(t) = -\frac{\pi_{12}(x_1 - \gamma)^{-1}}{\pi_{11} + \pi_{22} - 2\pi_{12}} \int_{x_1}^{A_1(t)}\left(\frac{x}{x_1}\right)^{-\theta_1-1}\left(\frac{1 - x}{1 - x_1}\right)^{-\theta_2-1}\left(\frac{x - \gamma}{x_1 - \gamma}\right)^{\theta_1 + \theta_2}dx.
\end{equation*}
Here, $\theta_1=\pi_{12}/(\pi_{22}-\pi_{12})$ and $\theta_2=\pi_{12}/(\pi_{11}-\pi_{12})$. The stability analysis of $A_1$ is then carried out simply using (\ref{foraintro}), and we get an explicit formula for the mating pattern. As an application of this formula, we show in Theorem \ref{homhet} that an infinite population $x_1,x_2,y_1,y_2$ with $x_1 = y_1\in(0,1)$ is homogamous (resp.\ heterogamous) if $\pi_{11}+\pi_{22} > 2\pi_{12}$ (resp.\ $\pi_{11}+\pi_{22} < 2\pi_{12}$), which is consistent with \eqref{ortakavm}, but this time in the infinite population setting and under the symmetry conditions. 

\subsection{Some remarks and open problems}

Several authors previously studied mating models that are similar to the ones in \cite{Gim88a}. See \cite{GunYil14a} for general references regarding pair formation models. One article of particular interest is \cite{Tay75}, where the ODE describing $Q(t)$ was given for two types and studied numerically. 

Panmixia is an important concept in population genetics. It is one of the main assumptions of  the Hardy-Weinberg law which states that genotype
frequencies remain constant in a population to which no evolutionary force acts on, see e.g.\,\cite[Chapter 1]{Ewe04}. In the literature, panmixia is also referred to as ``random mating". However, this term is obviously misleading since mating can be random yet assortative. Moreover, this confusion is even greater for a bottom-up approach such as in SEM, where ``random mating'' suggests that there are no preferences. Indeed, we show in Theorem \ref{thmfb} that there are instances where the mating pattern exhibits zero correlations between female and male types even though there are non-trivial preferences.

In the case of assortative mating, the genotype frequencies might differ greatly from the ones predicted by the Hardy-Weinberg law, see \cite[Chapter 4]{Gil98} and the references therein. Moreover, assortative mating is one of the key concepts of sexual selection, that is, the evolutionary force driven by mating. In the sexual selection literature, most models of pair formation assume that females unilaterally accept or reject males.  Various consequences of female choice have been studied in, e.g., \cite{Kir82,Lan81}. Observe that in the SEM model there is no specification of which sex makes the choice. Actually, this is an advantage of the model: unilateral decisions and choosiness can be incorporated into SEM by appropriately tuning the parameters, while retaining certain degrees of freedom that can be exploited for the purpose of finding exact formulas. However, to enable a self-contained study of sexual selection through SEM we need to extend the model in various directions which we discuss next.

SEM is about permanent pair formation and can be seen as a model of monogamous mating of animals in one mating season. Then, one natural direction in which to extend this model is to change the permanent pair structure. A simple way to do this would be to let the pairs separate with a certain rate and send the individuals that form it back to the singles pool. The life-time of a pair corresponds to ``latency" in the biological context. These kinds of models are important in the study of the evolution of female choice and the mutual evolution of female and male choices (via certain payoff functions for staying together depending on types -- see \cite{Alex1} and \cite{CEFGR}, respectively) and also suitable for studying sexually transmitted diseases (see \cite{DieHad88}). SEM can be generalized also by introducing polygamy, with each male having a limited number of mates (see \cite{Lee08} for such a model in a simpler setting). Finally, adding offspring production might lead to more general Lotka-Volterra systems.

The pair-type process of the Poisson EM model is density dependent, albeit approximately. Fluid and diffusion limits were first established for such processes by Kurtz \cite{K70,K78}. However, to the best of our knowledge, none of the general results in the literature directly cover our model (see Remark \ref{tamdegilya} for details). It is for this reason that we provide a self-contained proof of the fluid limit (Theorem \ref{thmLLN}). One can similarly try to establish a functional central limit theorem (CLT) for the pair-type process and then a CLT for the mating pattern which would complement the LLN (Theorem \ref{cormat}). This is one of our ongoing projects.

In Section \ref{sec4}, we follow the three-step procedure outlined in \eqref{repin}--\eqref{qode} and obtain a formula for $Q(t)$ in the symmetric $2\times 2$ case where the replicator equation constituting the first step is one-dimensional. One can attempt to follow the same procedure in (i) the general $2\times 2$ case and (ii) the symmetric $3\times 3$ case. Phase portraits of all Lotka-Volterra equations on the plane, hence of all two-dimensional replicator equations with constant intrinsic growth rates, are given in \cite{Bom}, which suggests that it might be possible to get an exact formula for the mating pattern in these two cases, too. However, much less is known about Lotka-Volterra equations in higher dimensions. In particular, numerical simulations show that the behavior in higher dimensions is chaotic and the type of chaos they exhibit is not understood at all. See \cite{Glp} for an example of chaos in three dimensions.

\section{Fluid limit and LLN}\label{sec2}

\subsection{Fluid limit of the pair-type process}\label{subsecFL}
The state space of the rescaled pair-type process $n^{-1}Q\n(t)$ is 
$$\mathcal{E}_n:=\left\{M\in\mathcal{M}^{k\times k}(n^{-1}\N\cup\{0\}):\; M_{i,\cdot}\leq n^{-1}x_i^{(n)},M_{\cdot,j}\leq n^{-1}y_j^{(n)},\forall i,j\in[k]\right\}.$$
Define $F^{(n)}=\left(F_{ij}^{(n)}\right)_{i,j\in[k]}:\mathcal{E}_n\to\mathcal{M}^{k\times k}([0,\infty))$
by
$$F_{ij}^{(n)}(M):=\left\{\begin{array}{ll}\frac{\pi_{ij}\big(n^{-1}x_i^{(n)}-M_{i,\cdot}\big)\big(n^{-1}y_j^{(n)}-M_{\cdot,j}\big)}{1-M_{tot}}& \text{if}\ M_{tot} < 1,\\ 0& \text{if}\ M_{tot} = 1.\end{array}\right.$$
We can rewrite the transition rates of $Q\n$, given in (\ref{trates}), as
\begin{equation*}
\rho^{(n)}(M,M+I^{ij})=nF_{ij}^{(n)}(n^{-1}M).
\end{equation*}
Consequently, we have the following representation (see \cite[Section 6.4]{EK}):
$$Q_{ij}\n(t)=J_{ij}\Big(n\int_0^t F_{ij}\n(n^{-1}Q\n(s)) ds\Big).$$
Here, $\big\{J_{ij}:\;i,j\in[k]\big\}$ is a collection of independent standard Poisson processes defined on a common probability space $(\Omega,\mathcal{F},\PP)$. Therefore, $n^{-1}Q\n$ is defined for all $n\in\N$ on the same probability space, too.

The following theorem establishes the fluid limit of the pair-type process, where the limiting (deterministic) process takes values in 
$$\mathcal{E}:=\left\{M\in\mathcal{M}^{k\times k}([0,\infty)):\; M_{i,\cdot}\leq x_i,M_{\cdot,j}\leq y_j,\forall i,j\in[k]\right\}$$
and satisfies a system of ODEs involving
$F=(F_{ij})_{i,j\in[k]}:\mathcal{M}^{k\times k}([0,\infty))\to\mathcal{M}^{k\times k}([0,\infty))$
which is defined by
\be\label{fij}
F_{ij}(M):=\left\{\begin{array}{ll}\frac{\pi_{ij}\big(x_i-M_{i,\cdot}\big)\big(y_j-M_{\cdot,j}\big)}{1-M_{tot}}& \text{if}\ M_{tot} \not= 1,\\ 0& \text{if}\ M_{tot} = 1.\end{array}\right.
\ee

\begin{theorem}\label{thmLLN}
	There exists a function $Q=(Q_{ij})_{i,j\in[k]}:[0,\infty)\to \mathcal{E}$ satisfying 
	\be\label{ode}
	Q(t)=\int_0^t F(Q(s))ds,
	\ee
	and for any $T\in[0,\infty)$, $\mathbb{P}$-a.s. 
	$$\lim_{n\to\infty} \sup_{0\leq t\leq T}\left|n^{-1}Q\n(t)-Q(t)\right|=0.$$
\end{theorem}

Since $F\n$ and $F$ are close (in an appropriate sense which is made precise below) when $n$ is large, $Q\n$ is approximately a density dependent population process (see \cite[Chapter 11]{EK}). Fluid limits were first obtained for such processes by Kurtz in \cite{K70} (showing convergence in probability) and then in \cite{K78} (showing almost sure convergence). The proof of Theorem \ref{thmLLN} is adapted from the latter work, but it involves some modifications (see Remark \ref{tamdegilya}). Before presenting the proof, we give two lemmas.


\begin{lemma}\label{lemferdi}
Let $n\in\mathbb{N}$ and $i,j,i',j'\in [k]$.
\begin{itemize}
\item [(a)] For every $M\in\mathcal{E}_n$ and $M'\in\mathcal{E}$,
$$0 \le F_{ij}\n(M)\leq n^{-1}\pi_{ij}\left(x_i\n\wedge y_j\n\right) \le \pi_{ij}\quad\text{and}\quad 0 \le F_{ij}(M') \leq \pi_{ij}\left(x_i\wedge y_j\right) \le \pi_{ij}.$$
\item [(b)] For every $M\in\mathcal{M}^{k\times k}([0,\infty))$ with $M_{tot} < 1$,
$$\frac{\partial F_{ij}(M)}{\partial M_{i'j'}} = \pi_{ij}\left[\left(\frac{x_i - M_{i,\cdot}}{1-M_{tot}}\right)\left(\frac{y_j - M_{\cdot,j}}{1-M_{tot}}\right) - \left(\frac{x_i - M_{i,\cdot}}{1-M_{tot}}\right)\delta_{jj'} - \left(\frac{y_j - M_{\cdot,j}}{1-M_{tot}}\right)\delta_{ii'}\right],$$
where $\delta_{ij}$ denotes the Kronecker delta function. In particular, 
\begin{equation}\label{inpabir}
\left|\frac{\partial F_{ij}(M)}{\partial M_{i'j'}}\right| \le \begin{cases}\quad\ \pi_{ij} &\ \text{if $M_{tot} < 1$ and $M\in\mathcal{E}$,}\\ \frac{3\pi_{ij}}{(1-M_{tot})^2}&\ \text{if $M_{tot} < 1$ and $M\notin\mathcal{E}$}.\end{cases}
\end{equation}
\item [(c)] For every $M\in\mathcal{E}_n$ with $M_{tot} < 1$,
\begin{align*}F_{ij}^{(n)}(M) - F_{ij}(M) = \frac{\pi_{ij}}{1-M_{tot}}&\left[\left(n^{-1}x_i\n - x_i\right)\left(n^{-1}y_j\n - M_{\cdot,j}\right)\right.\\
&\qquad\ \left.+ \left(n^{-1}y_j\n - y_j\right)\left(x_i - M_{i,\cdot}\right)\right].
\end{align*}
In particular,
\begin{equation}\label{inpaiki}
\left|F_{ij}^{(n)}(M) - F_{ij}(M)\right| \le
\begin{cases}
\pi_{ij}\left[\left|n^{-1}x_i\n - x_i\right| + \left|n^{-1}y_j\n - y_j\right|\right]&\ \text{if $M\in\mathcal{E}_n\cap\mathcal{E}$},\\
\frac{\pi_{ij}}{1-M_{tot}}\left[\left|n^{-1}x_i\n - x_i\right| + \left|n^{-1}y_j\n - y_j\right|\right]&\ \text{if $M\in\mathcal{E}_n\setminus\mathcal{E}$}.
\end{cases}
\end{equation}
\end{itemize}
\end{lemma}

\begin{proof}
	Verification of these simple equalities and bounds is left to the reader. 
\end{proof}

\begin{lemma}\label{lemsonMC}
For every $T\in[0,\infty)$ and $c > \overline\pi:=\max_{i,j\in[k]}\pi_{ij}$, 
$$\mathbb{P}\left(1- n^{-1}Q_{tot}\n(T) \ge e^{-cT}\ \text{for sufficiently large $n$}\right) = 1.$$
\end{lemma}

\begin{proof}
Assume without loss of generality that $p_{ij} = \frac{\pi_{ij}}{\overline{\pi}}$, $\alpha_i \equiv 0$ and $\beta_j \equiv \overline{\pi}$. In particular, only males fire. Fix $T\in[0,\infty)$ and let $R\n(T)$ be the number of males who have never fired by time $T$. Since $n- Q_{tot}\n(T)$ is the number of males who are single by time $T$,
\be\label{uczkurt}
n- Q_{tot}\n(T) \ge R\n(T).
\ee
Enumerate the males and let
$$\xi_m(T) = \begin{cases}1&\ \text{if the $m$th male has never fired by time $T$},\\0&\ \text{else}.\end{cases}$$
Then, $(\xi_m(T))_{m\in[n]}$ are independent Bernoulli trials with $P(\xi_m(T) = 1) = e^{-\overline{\pi}T}$. Fix $c > \overline{\pi}$.
Since $R\n(T) = \sum_{m=1}^n \xi_m(T)$, a standard application of the exponential Chebyshev inequality shows that $\mathbb{P}(n^{-1}R\n(T) < e^{-cT})\to 0$ exponentially as $n\to\infty$. By the Borel-Cantelli lemma,
$$\mathbb{P}\left(n^{-1}R\n(T) \ge e^{-cT}\ \text{for sufficiently large $n$}\right) = 1.$$
In combination with \eqref{uczkurt}, this implies the desired result.
\end{proof}

\begin{proof}[Proof of Theorem \ref{thmLLN}]
Since $F$ is bounded and Lipschitz continuous on $\mathcal{E}$ by Lemma \ref{lemferdi}(a,b), the system of ODEs in (\ref{ode}) has a unique solution $Q$. Let us show that this solution exists for all times. By our assumptions in the Introduction, $\pi_{ij} = p_{ij}(\alpha_i + \beta_j) > 0$ for all $i,j\in[k]$. Thus, $\underline{\pi}:=\min_{i,j\in[k]}\pi_{ij}>0$. Recalling $\overline\pi:=\max_{i,j\in[k]}\pi_{ij}$ and (\ref{fij}), we get
\begin{equation*}
\underline{\pi}(1-Q_{tot}(t))\leq \prt{Q_{tot}}\leq \overline{\pi}(1-Q_{tot}(t)).
\end{equation*}
Since $Q_{tot}(0)=0$, this implies
\be\label{lips}
1-e^{-\underline{\pi}t}\leq Q_{tot}(t)\leq 1-e^{-\overline{\pi}t}.
\ee
Thus, $Q_{tot}(t)<1$ for any $t\in[0,\infty)$, and in particular, $Q$ exists for all times.

The difference between the rescaled pair-type process and its prospective limit $Q$ can be controlled as follows. For every $i,j\in[k]$ and $t\in[0,T]$,
\begin{align}
\left|n^{-1}Q_{ij}\n(t) - Q_{ij}(t)\right| &= \left|n^{-1}J_{ij}\Big(n\int_0^t F_{ij}\n(n^{-1}Q\n(s)) ds\Big) - \int_0^t F_{ij}(Q(s))ds\right|\nonumber\\
&\le \left|n^{-1}J_{ij}\Big(n\int_0^t F_{ij}\n(n^{-1}Q\n(s)) ds\Big) - \int_0^t F_{ij}\n(n^{-1}Q\n(s))ds\right|\label{term1}\\
&\qquad + \left|\int_0^t F_{ij}\n(n^{-1}Q\n(s))ds - \int_0^t F_{ij}(n^{-1}Q\n(s))ds\right|\label{term2}\\
&\qquad + \left|\int_0^t F_{ij}(n^{-1}Q\n(s))ds - \int_0^t F_{ij}(Q(s))ds\right|.\label{term3}
\end{align}
It follows from Lemma \ref{lemferdi}(a) that the term in \eqref{term1} is bounded from above by
$$a_{ij}\n(t) := \sup_{0\le u\le \pi_{ij}t}\left|n^{-1}J_{ij}(nu) - u\right| \le a_{ij}\n(T).$$
Fix $c > \overline{\pi}$. Lemma \ref{lemferdi}(c) and Lemma \ref{lemsonMC} imply that, on some $\Omega_{T}\in\mathcal{F}$ with $\mathbb{P}(\Omega_T) = 1$, the term in \eqref{term2} is bounded from above by
$$b_{ij}\n(t) := \pi_{ij}te^{cT}\left[\left|n^{-1}x_i\n - x_i\right| + \left|n^{-1}y_j\n - y_j\right|\right] \le b_{ij}\n(T)$$
for sufficiently large $n$. Similarly, using Lemma \ref{lemferdi}(b), Lemma \ref{lemsonMC} and \eqref{lips}, on the set $\Omega_T$, the term in \eqref{term3} is bounded from above by
$$3k^2\pi_{ij}e^{2cT}\int_0^t\left|n^{-1}Q\n(s) - Q(s)\right|ds$$
for sufficiently large $n$. (Recall from the Introduction that $|M|=\max_{i,j\in[k]}|M_{ij}|$.) Therefore,
$$\left|n^{-1}Q\n(t) - Q(t)\right| \le \max_{i,j\in[k]}\left(a_{ij}\n(T) + b_{ij}\n(T)\right) +  3k^2\overline\pi e^{2cT}\int_0^t\left|n^{-1}Q\n(s) - Q(s)\right|ds.$$
Since $Q\n(0) = Q(0) = 0$,
$$\left|n^{-1}Q\n(t) - Q(t)\right| \le \max_{i,j\in[k]}\left(a_{ij}\n(T) + b_{ij}\n(T)\right)\exp\left(3k^2\overline\pi te^{2cT}\right)$$
on $\Omega_T$ by Gronwall's inequality. For every $i,j\in[k]$, a standard application of Doob's martingale inequality (followed by the Borel-Cantelli lemma) shows that $$\mathbb{P}\left(\lim_{n\to\infty}a_{ij}\n(T) = 0\right) = 1.$$
Since $\lim_{n\to\infty}b_{ij}\n(T) = 0$ by our assumption in \eqref{convxy}, we are done.
\end{proof}

\begin{remark}\label{tamdegilya}
The existence of the fluid limit is stated and proved in \cite[Theorem 2.2]{K78} for a wide class of density dependent population processes. In order for this class to contain $Q\n$, two conditions would have to be satisfied:
\begin{itemize}
	\item [(i)] $F$ is Lipschitz continuous on $\mathcal{E}_n\cup\mathcal{E}$; and
	\item [(ii)] $|F\n(M) - F(M)| = O(n^{-1})$ uniformly for $M\in\mathcal{E}_n$.
\end{itemize}
However, as we have seen in \eqref{inpabir} and \eqref{inpaiki}, these conditions fail to hold in general since $\mathcal{E}_n$ need not be a subset of $\mathcal{E}$. In the proof of Theorem \ref{thmLLN}, we resolved these issues with the help of Lemma \ref{lemsonMC} (which allowed us to restrict the analysis to the region where $M_{tot}$ is bounded away from $1$) and the observation that the error in the second condition can be relaxed to $o(1)$.
\end{remark}

\subsection{LLN for the mating pattern}

We first describe the state space of the rescaled mating pattern $n^{-1}Q\n(T_n)$ and its asymptotic counterpart. Define
\begin{align*}
\mathcal{E}'_n &:=\left\{M\in\mathcal{M}^{k\times k}(n^{-1}\N\cup\{0\}):\; M_{i,\cdot} = n^{-1}x_i^{(n)},M_{\cdot,j} = n^{-1}y_j^{(n)},\forall i,j\in[k]\right\}\ \text{and}\\
\mathcal{E}' &:=\left\{M\in\mathcal{M}^{k\times k}([0,\infty)):\ M_{i,\cdot}= x_i,M_{\cdot,j}= y_j,\forall i,j\in[k]\right\}.
\end{align*}
By its definition, at time $T_n$ there are no singles left and thus, $n^{-1}Q\n(T_n)\in \mathcal{E}'_n\subset \mathcal{E}_n$. Also note that, for $M\in\mathcal{E}$, $F(M)=0$ if and only if $M\in\mathcal{E}'$. As a result, using (\ref{lips}), we can conclude that $\lim_{t\to \infty}Q(t)=:Q(\infty)$ exists and $Q(\infty)\in \mathcal{E}'\subset\mathcal{E}$.

The following result extends the fluid limit of the pair-type process (Theorem \ref{thmLLN}) to a LLN for the mating pattern.

\begin{theorem}\label{cormat}
$\mathbb{P}$-a.s.
\begin{equation*}
\lim_{n\to\infty}n^{-1}Q\n(T_n) = Q(\infty).
\end{equation*}
\end{theorem}

\begin{proof}

We define
\begin{equation*}
T^\delta:=\inf\big\{t\geq 0:\; Q_{tot}(t)\geq 1-\delta\big\},\quad \delta>0.
\end{equation*}
By (\ref{lips}), we have $T^\delta<\infty$, and $T^\delta\to \infty$ as $\delta\to 0$. Also, it is clear that 
$$\left|Q(T^\delta)-Q(\infty)\right|\leq \delta.$$
Now we define the corresponding stopping time for the Markov process $Q\n$ by
\begin{equation*}
T_n^\delta:=\inf\{t\geq 0:\; n^{-1}Q_{tot}\n(t)\geq 1-\delta\},\quad \delta>0.
\end{equation*}
Then, since obviously $T_n\geq T_n^\delta$ for any $\delta>0$ and $n\geq 1$, we have
$$\left|n^{-1}Q\n(T_n) - n^{-1}Q\n(T_n^\delta)\right|\leq \delta.$$
The triangle inequality gives
\begin{align*}
&\Big|n^{-1}Q\n(T_n)-Q(\infty)\Big|\\ &\quad\leq \Big| n^{-1}Q\n(T_n) - n^{-1}Q\n(T_n^\delta)\Big|+\Big|n^{-1}Q\n(T_n^\delta) - Q(T^\delta)\Big| + \Big|Q(T^\delta)-Q(\infty)\Big|\\
&\quad\le \Big|n^{-1}Q\n(T_n^\delta) - Q(T^\delta)\Big| + 2\delta.
\end{align*}
Therefore, the desired result will follow once we prove that $\PP$-a.s.
\be\label{arauc}
\lim_{n\to\infty}n^{-1}Q\n(T_n^\delta) = Q(T^\delta).
\ee

Fix $\delta>0$. For any $\epsilon<\delta$ we have again $T^{\delta-\epsilon}<\infty$. Thus, via Theorem \ref{thmLLN}, $\PP$-a.s., for all sufficiently large $n$,
\begin{equation*}
n^{-1}Q\n_{tot}(T^{\delta-\epsilon})\geq Q_{tot}(T^{\delta-\epsilon})-\epsilon/2= 1-\delta+\epsilon/2>1-\delta.
\end{equation*}
Hence, $\PP$-a.s., $\limsup _{n\to \infty}T_n^\delta\leq T^{\delta-\epsilon}$. 
Now we use Theorem \ref{thmLLN} on the time interval $[0,T^\delta]$. $\PP$-a.s., for all sufficiently large $n$, and for $t\leq T^\delta$ with $n^{-1}Q_{tot}\n(t)\geq 1-\delta$,
\begin{equation*}
Q_{tot}(t)\geq n^{-1}Q_{tot}\n(t) - \epsilon/2\geq 1-\delta-\epsilon/2>1-\delta-\epsilon.
\end{equation*}
Thus, $t\geq T^{\delta+\epsilon}$ for any such $t$. Also, for any $t>T^\delta$, since $T^\delta\geq T^{\delta+\epsilon}$, we have $t>T^{\delta+\epsilon}$. Hence, $\PP$-a.s., for all sufficiently large $n$, and $t\geq 0$ with $n^{-1}Q_{tot}\n(t)\geq 1-\delta$, we have $t\geq T^{\delta+\epsilon}$, that is, $\liminf_{n\to\infty} T_n^\delta\geq T^{\delta+\epsilon}$. Since $Q_{tot}$ is continuous and increasing, as $\epsilon\to 0$, both $T^{\delta+\epsilon}\to T^{\delta} $ and $T^{\delta-\epsilon}\to T^\delta$. Therefore, $\PP$-a.s.
$$\lim_{n\to\infty}T_n^\delta\to T^\delta.$$
Hence, for any $\epsilon'>0$ given, $\PP$-a.s., for all sufficiently large $n$, we have $T^\delta-\epsilon'\leq T^\delta_n\leq T^\delta +\epsilon'$. Since $Q_{ij}\n(t)$ is non-decreasing in $t$ for any $i,j\in[k]$,
\begin{equation*}
n^{-1}Q\n_{ij}(T^\delta-\epsilon') - Q_{ij}(T^\delta)\leq n^{-1}Q\n_{ij}(T_n^\delta) - Q_{ij}(T^\delta)\leq n^{-1}Q\n_{ij}(T^\delta+\epsilon') - Q_{ij}(T^\delta).
\end{equation*} 
Via the inequalities
\begin{equation*}
\Big|n^{-1}Q\n_{ij}(T^\delta-\epsilon') - Q_{ij}(T^\delta)\Big|\leq \Big|n^{-1}Q\n_{ij}(T^\delta-\epsilon') - Q_{ij}(T^\delta-\epsilon')\Big|+\Big|{Q_{ij}(T^\delta-\epsilon')}-Q_{ij}(T^\delta)\Big|
\end{equation*}
and
\begin{equation*}
\Big|n^{-1}Q\n_{ij}(T^\delta+\epsilon') - Q_{ij}(T^\delta)\Big|\leq \Big|n^{-1}Q\n_{ij}(T^\delta+\epsilon') - Q_{ij}(T^\delta+\epsilon')\Big|+\Big|{Q_{ij}(T^\delta+\epsilon')}-Q_{ij}(T^\delta)\Big|,
\end{equation*}
using once again Theorem \ref{thmLLN} and the continuity of $Q$, \eqref{arauc} follows and we are done.
\end{proof}

\section{Analysis of the fluid limit}\label{sec3}

\subsection{Lotka-Volterra and replicator equations}
Recall from Section \ref{subsecoverview} that 
\begin{equation*}
X_i(t)=x_i-Q_{i,\cdot}(t),\quad Y_j(t)=y_j-Q_{\cdot,j}(t),\quad Z(t)=1-Q_{tot}(t)
\end{equation*}
denote the density of type-$i$ single females, type-$j$ single males, and all single females (or males), respectively. We have also introduced
\begin{equation*}
A_i(t)=\frac{X_i(t)}{Z(t)}\quad \mbox{and}\quad B_{j}(t)=\frac{Y_j(t)}{Z(t)}.
\end{equation*}
In words, $A_i$ is the fraction of type-$i$ females among all single females, and $B_j$ is the fraction of type-$j$ males among all single males. Then, for any $t\geq 0$,
\begin{equation*}
A_1(t)+\cdots+A_k(t)=B_1(t)+\cdots +B_k(t)=1.
\end{equation*}
To state our next result, we define a $2k\times 2k$ matrix 
\begin{equation*}
\hat \Pi := \left(\begin{array}{ll} 0 & \Pi\\ \Pi^T & 0\end{array}\right)
\end{equation*}
as well as vector-valued functions
$$U(t):=\big(X_1(t),\dots,X_k(t),Y_1(t),\dots,Y_k(t)\big)\ \ \text{and}\ \ C(t):=\frac{1}{2}\big(A_1(t),\dots,A_k(t),B_{1}(t),\dots,B_k(t)\big).$$

\begin{theorem}\label{ltrep}
\begin{itemize}
	\item [(a)] $U$ satisfies
	\be\label{lt}
	{\prt{U_i}}=-\frac{1}{Z(t)}U_i(t)(\hat \Pi U(t))_i,\quad i\in[2k],\ t\in(0,\infty),
	\ee
	that is, up to a time change, $U$ is the solution of a system of Lotka-Volterra equations. 
	\item [(b)] $C$ satisfies the following system of replicator equations:
	\be\label{rep}
	\prt{C_i}=-2C_i(t)\big[ (\hat \Pi C(t))_i- C^T(t)\hat \Pi C(t)\big],\quad i\in[2k],\ t\in(0,\infty).
	\ee
\end{itemize}
\end{theorem}

\begin{remark}
When the matrix $\Pi$ is symmetric, which means that its entries do not depend on the sexes but only on the types, and if $x_i=y_i$ for all $i\in[k]$, it is clear that $X_i(t)= Y_i(t)$  and $A_i(t)=B_i(t)$, for all $i\in[k]$ and $t\geq 0$. Consequently, the $2k$ replicator equations in (\ref{rep}) simplify to the following replicator system with $k$ variables: 
\begin{equation}\label{rep2}
\prt{A_i}=-A_i(t)\big[(\Pi A(t))_i-A^T(t)\Pi A(t)\big],\quad i\in[k],\ t\in(0,\infty).
\end{equation}
We use this observation in Section \ref{sec4} while studying the symmetric $2\times 2$ case. A similar simplification also applies to the Lotka-Volterra equations in (\ref{lt}).
\end{remark}

\begin{proof}[Proof of Theorem \ref{ltrep}]
Let us write $X=(X_1,\dots,X_k)$ and $Y=(Y_1,\dots,Y_k)$. Using (\ref{fij}), (\ref{ode}), and the definitions of $X_i,Y_j$ and $Z,$ we get
\be\label{bangla}
\prt{Q_{ij}}=\frac{\pi_{ij}X_i(t)Y_j(t)}{Z(t)}.
\ee
Thus, for $i\in[k]$
\be\label{xder}
\prt{U_i}=\prt{X_i}=-\sum_{j=1}^k \prt{Q_{ij}}=-\frac{1}{Z(t)}X_i(t)\sum_{j=1}^k\pi_{ij}Y_{j}(t)=-\frac{1}{Z(t)}U_i(t)(\Pi Y(t))_i.
\ee
Similarly, for $j\in[k]$
\be\label{yder}
\prt{U_{k+j}}=\prt{Y_j}=-\frac{1}{Z(t)}U_{k+j}(t)(\Pi^T X(t))_j.
\ee
Hence, noting that $(\hat \Pi U)_i=(\Pi Y)_i$ and $(\hat \Pi U)_{k+j}=(\Pi^T X)_j$ for $i,j\in[k]$ gives (\ref{lt}).

Summing (\ref{xder}) over $i$ (or equivalently (\ref{yder}) over $j$) and using the definitions of $A_i$ and $B_j$, we get
\be\label{forz}
\prt{Z}=-Z(t)\big(A^T(t)\Pi B(t)\big)=-Z(t)\big(B^T(t)\Pi^T A(t)\big).
\ee
As a result, using (\ref{xder}), for $i\in[k]$
\be\label{fora}
\begin{aligned}
2\prt{C_i}=\prt{A_i}&=\prt{X_i}\frac{1}{Z(t)}-\frac{X_i(t)}{Z^2(t)}\prt{Z}
\\&=-A_i(t)\left[(\Pi B(t))_i -A^T(t)\Pi B(t)\right].
\end{aligned}
\ee
Similarly, using (\ref{yder}), for $j\in[k]$
\be\label{forb}
\quad 2\prt{C_{k+j}}=\prt{B_j}=-B_j(t)\left[(\Pi^T A(t))_j -B^T(t)\Pi^T A(t)\right].
\ee
By the definition of $\hat \Pi$ we have
\begin{equation*}
(\hat \Pi C(t))_i=\frac{1}{2}(\Pi B(t))_i,\quad (\hat \Pi C(t))_{k+j}=\frac{1}{2}(\Pi^T A(t))_j,\quad i,j\in[k],
\end{equation*}
and 
\begin{equation*}
C^T(t)\hat \Pi C(t)=\frac{1}{4}A^T(t)\Pi B(t)+\frac{1}{4}B^T(t) \Pi^T A(t)=\frac{1}{2}A^T(t)\Pi B(t)=\frac{1}{2} B^T(t) \Pi^T A(t).
\end{equation*}
Thus, using (\ref{fora}), for $i\in [k]$
\begin{equation*}
\prt{C_i}=\frac{1}{2}\prt{A_i}=-2C_i(t)\big[ (\hat \Pi C(t))_i- C^T(t)\hat \Pi C(t)\big],
\end{equation*}
and using (\ref{forb}), for $j\in[k]$
\begin{equation*}
\prt{C_{k+j}}=\frac{1}{2}\prt{B_j}=-2C_{k+j}(t)\big[ (\hat \Pi C(t))_{k+j}- C^T(t)\hat \Pi C(t)\big].
\end{equation*}
This concludes the proof of (\ref{rep}).
\end{proof}

\subsection{Exact solution under fine balance}\label{sec32}

As we have mentioned in Section \ref{subsecprev}, in \cite{GunYil14a} we proved that the fine balance condition (given in Definition \ref{defnfineb}) characterizes panmixia for the species (in the context of finite populations). The next theorem considers infinite populations and gives explicit formulas for the solution of the system of replicator equations and for the pair-type process under the fine balance condition. 
\begin{theorem}\label{thmfb}
Assume that the fine balance condition (\ref{fb}) is satisfied. Then
\begin{equation*}
A_i(t)=\frac{x_i e^{-\bar\alpha_i t}}{\sum_{i'} x_{i'} e^{-\bar\alpha_{i'} t}},\quad B_j(t)=\frac{y_j e^{-\bar\beta_j t}}{\sum_{j'} y_{j'}e^{-\bar\beta_{j'} t}}
\end{equation*}
and
\begin{equation*}
Q_{ij}(t)=x_i y_j (1-e^{-\pi_{ij}t}).
\end{equation*}
In particular,
\begin{equation*}
Q_{ij}(\infty)=x_i y_j.
\end{equation*}
\end{theorem}

\begin{remark}
	The formulas in Theorem \ref{thmfb} can also be obtained from \cite[Theorem 3.6]{GunYil14a} via the fluid limit (Theorem \ref{thmLLN}) and the dominated convergence theorem. However, our method here is completely different and self-contained.
\end{remark}

\begin{proof}[Proof of Theorem \ref{thmfb}]
Using (\ref{fora}), for $i\in[k]$ we get
\be\label{tha}
\begin{aligned}
\frac{d }{d t}\log \big(A_i(t)/A_1(t)\big)&=\prt{\log A_i}-\prt{\log A_1}
\\&=-\left[(\Pi B(t))_i- A^T(t)\Pi B(t)\right]+\left[(\Pi B(t))_1- A^T(t)\Pi B(t)\right]
\\&=-\left[(\Pi B(t))_i-(\Pi B(t))_1\right].
\end{aligned}
\ee
Similarly, by (\ref{forb}), for $j\in[k]$ we have
\be\label{thb}
\frac{d }{d t}\log \big(B_j(t)/B_1(t)\big)=-\left[(\Pi^T A(t))_j-(\Pi^T A(t))_1\right].
\ee
Using (\ref{fb}), for $i\in[k]$ we get
\begin{equation*}
(\Pi B(t))_i=\sum_{j=1}^k \pi_{ij} B_j(t)=\sum_{j=1}^k (\bar\alpha_i+\bar\beta_j) B_j(t)=\bar\alpha_i+\sum_{j=1}^k \bar\beta_{j} B_j(t).
\end{equation*}
Then, (\ref{tha}) yields
\begin{equation*}
\frac{d }{d t}\log \big(A_i(t)/A_1(t)\big)=-(\bar\alpha_i-\bar\alpha_1).
\end{equation*}
Hence,
\begin{equation*}
\frac{A_i(t)}{A_1(t)}=\frac{A_i(0)}{A_1(0)}e^{-(\bar\alpha_i-\bar\alpha_1)t} = \frac{x_ie^{-\bar{\alpha}_it}}{x_1e^{-\bar{\alpha}_1t}}.
\end{equation*}
Finally, since $A_1(t)+\cdots +A_k(t)=1$, we get
\begin{equation*}
A_i(t)=\frac{x_i e^{-\bar\alpha_i t}} {\bar{A}(t)},\quad\text{where}\quad \bar{A}(t)=\sum_{i'=1}^k x_{i'} e^{-\bar\alpha_{i'} t}
\end{equation*}
is the normalization term. Similarly, using \eqref{fb} and (\ref{thb}), we get
\begin{equation*}
B_j(t)=\frac{y_j e^{-\bar\beta_j t}} {\bar{B}(t)},\quad\text{where}\quad \bar{B}(t)=\sum_{j'=1}^k y_{j'} e^{-\bar\beta_{j'} t}.
\end{equation*}

Next, we compute $Z(t)$. Note that we can use (\ref{forz}) to write 
\begin{align*}
\frac{d }{d t}\log Z(t) &= - A^T(t)\Pi B(t) = - \sum_{i=1}^k\sum_{j=1}^k(\bar\alpha_i +\bar \beta_j)A_i(t)B_j(t)\\
&= - \sum_{i=1}^k\sum_{j=1}^k\bar\alpha_iA_i(t)B_j(t) - \sum_{i=1}^k\sum_{j=1}^k\bar\beta_jA_i(t)B_j(t)\\
&= - \sum_{i=1}^k\bar\alpha_iA_i(t) - \sum_{j=1}^k\bar\beta_jB_j(t) = - \sum_{i=1}^k\frac{ x_i\bar\alpha_ie^{ - \bar\alpha_it}}{\bar A(t)} - \sum_{j=1}^k\frac{ y_j\bar\beta_je^{ - \bar\beta_jt}}{\bar B(t)}\\
&= \frac{1}{\bar{A}(t)}\prt{\bar A} + \frac{1}{\bar{B}(t)}\prt{\bar B} = \frac{d}{dt}\log \bar{A}(t) + \frac{d}{dt}\log \bar{B}(t) = \frac{d}{dt}\log \left[\bar{A}(t)\bar{B}(t)\right].
\end{align*}
Since $\bar{A}(0) = \bar{B}(0) = Z(0) = 1$, we deduce that
$$Z(t) = \bar{A}(t)\bar{B}(t) = \sum_{i=1}^k x_ie^{ - \bar\alpha_it}\sum_{j=1}^k y_je^{ - \bar\beta_jt} = \sum_{i=1}^k\sum_{j=1}^k x_i y_je^{ - \pi_{ij}t}.$$

Finally, we compute $Q_{ij}(t)$. We can use (\ref{bangla}) to write
$$\prt{Q_{ij}} = \pi_{ij}Z(t)A_i(t)B_j(t) = \pi_{ij}\bar{A}(t)\bar{B}(t)\frac{ x_ie^{ - \bar\alpha_it}}{\bar{A}(t)}\frac{ y_je^{ - \bar\beta_jt}}{\bar{B}(t)} = \pi_{ij} x_i y_je^{ - \pi_{ij}t}.$$
Since $Q_{ij}(0) = 0$, we conclude that
$$Q_{ij}(t) = x_i y_j\left(1 - e^{ - \pi_{ij}t}\right).\qedhere$$
\end{proof}

\section{The symmetric $2\times 2$ case}\label{sec4}
\newcommand{\ppp}[1]{\dot{{#1}}}
In this section we use the shorthand notation $\dot f$ to denote the time derivative $\prt{f}$ of any function $f$. We assume that $k=2$, $\pi_{12}=\pi_{21}$ and $x_1=y_1$. Setting $A_2=1-A_1$, the replicator equation in (\ref{rep2}) becomes a one-dimensional ODE given by
\be\label{syfora}
\ppp{A}_1=-A_1(1-A_1)\Big[(\pi_{11}+\pi_{22}-2\pi_{12})A_1-(\pi_{22}-\pi_{12})\Big],
\ee
with $A_1(0)=x_1$, and (\ref{forz}) is equivalent to
\newcommand{\dt}{{(\pi_{11}+\pi_{22}-2\pi_{12})}}
\newcommand{\dtt}{{(\pi_{22}-\pi_{12})}}
\be\label{syforz}
\frac{\ppp{Z}}{Z}=-\dt A_1^2+2\dtt A_1-\pi_{22},
\ee
with $Z(0)=1$. We already solved for $Q$ in the previous section under the fine balance condition so we exclude that case here, i.e., we assume that $\pi_{11}+\pi_{22}\not= 2\pi_{12}$. Hence, setting 
\begin{equation*}
\gamma=\frac{\pi_{22}-\pi_{12}}{\pi_{11}+\pi_{22}-2\pi_{12}},
\end{equation*}
the equation in (\ref{syfora}) becomes
\be\label{gama}
\ppp{A}_1=-\dt A_1(1-A_1)(A_1-\gamma).
\ee
Recall that our goal is to find a formula for the mating pattern. 
Since $k=2$, it suffices to find a formula for $Q_{12}(\infty)$ because 
\begin{equation*}
Q_{11}(\infty)=x_1-Q_{12}(\infty),\quad Q_{21}(\infty)=y_1-Q_{11}(\infty),\quad \text{and}\quad Q_{22}(\infty)=x_2-Q_{21}(\infty).
\end{equation*}
For this we use (\ref{bangla}), which can be written in the form
\be\label{gamq}
\ppp{Q}_{12}=\pi_{12}ZA_1(1-A_1).
\ee
We first study the case $\gamma\in\{0,1\}$, that is, $\pi_{11}=\pi_{12}$ or $\pi_{22}=\pi_{12}$.

\subsection{$\boldsymbol{ \gamma\in\{0,1\}}$}
We first investigate the case $\gamma=1$, that is, $\pi_{11}=\pi_{12}$. 

\newcommand{\tp}{(t)}
Note that (\ref{gama}) and (\ref{syforz}) become, respectively,
\be\label{onefora}
\ppp{A}_1=(\pi_{22}-\pi_{12})A_1(1-A_1)^2
\ee
and
\be\label{kong}
\frac{\ppp{Z}}{Z}=-(\pi_{22}-\pi_{12})(1-A_1)^2-\pi_{12}.
\ee
We can use partial fractions to write (\ref{onefora}) as
$$\left(\frac1{A_1} + \frac1{1-A_1} + \frac1{(1 - A_1)^2}\right)\ppp{A}_1 = \pi_{22} - \pi_{12}.$$
Integrating both sides and using the initial condition  $A_1(0) = x_1$, we get
\begin{equation}\label{ookan}
\frac{(1-x_1)A_1(t)}{x_1(1 - A_1(t))}\exp\left\{\frac1{1-A_1(t)} - \frac1{1-x_1}\right\} = e^{(\pi_{22} - \pi_{12})t}.
\end{equation}
This is an implicit formula for $A_1(t)$.

Next, we find a formula for $Z(t)$. We know from (\ref{onefora}) that $$(\pi_{22} - \pi_{12})(1 - A_1)^2 = \frac{\dot{A}_1}{A_1}.$$
Substituting this in (\ref{kong}), we see that
$$\frac{\ppp{Z}}{Z} = - \frac{\ppp{A}_1}{A_1} - \pi_{12}.$$
Integrating both sides and using the initial condition $Z(0) = 1$, we get
\begin{equation}\label{zoziz}
Z(t) = \left(\frac{x_1}{A_1(t)}\right)e^{-\pi_{12}t}.
\end{equation}
We can express $Z(t)$ in terms of $A_1(t)$ only (i.e., without any explicit $t$ dependence.) Indeed, raising both sides of (\ref{ookan}) to power $-\theta_1$ where
$$\theta_1 := \frac{\pi_{12}}{\pi_{22} - \pi_{12}}$$
gives
$$\left(\frac{(1-x_1)A_1(t)}{x_1(1 - A_1(t))}\right)^{-\theta_1}\exp\left\{-\theta_1\left(\frac1{1-A_1(t)} - \frac1{1-x_1}\right)\right\} = e^{-\pi_{12}t}.$$
Plugging this into the right-hand side of (\ref{zoziz}), we get
\begin{equation}\label{beklos}
Z \tp= \left(\frac{1 - A_1\tp}{1 - x_1}\right)^{\theta_1}\left(\frac{A_1\tp}{x_1}\right)^{-\theta_1 - 1}\exp\left\{-\theta_1\left(\frac1{1-A_1\tp} - \frac1{1-x_1}\right)\right\}.
\end{equation}

Finally, we express $Q_{12}(t)$ in terms of $A_1(t)$. We put (\ref{onefora}) in the form $$A_1(1 - A_1) = \frac{\dot{A}_1}{(\pi_{22} - \pi_{12})(1-A_1)}.$$
We can use this and (\ref{beklos}) to write
\begin{align*}
\dot{Q}_{12} &= \pi_{12}ZA_1(1 - A_1) = \frac{\theta_1Z\dot{A}_1}{1 - A_1}\\
&= \frac{\theta_1}{1-x_1}\left(\frac{1 - A_1}{1 - x_1}\right)^{\theta_1 - 1}\left(\frac{A_1}{x_1}\right)^{-\theta_1-1}\exp\left\{{-}\theta_1\left(\frac1{1-A_1} - \frac1{1-x_1}\right)\right\}\dot{A}_1.
\end{align*}
Integrating both sides, using the initial conditions $A_1(0) = x_1$ and $Q_{12}(0) = 0$, and making a change of variables, we get
\begin{align*}
Q_{12}(t) & = \frac{\theta_1}{1-x_1}\int_{x_1}^{A_1(t)}\left(\frac{1 - x}{1 - x_1}\right)^{\theta_1 - 1}\left(\frac{x}{x_1}\right)^{-\theta_1-1}\exp\left\{{-}\theta_1\left(\frac1{1-x} - \frac1{1-x_1}\right)\right\}dx\\
& = x_1\theta_1\int_1^{\zeta(t)}x^{ - (\theta_1 +1)}e^{ - \left(\frac{x_1}{1 - x_1}\right)\theta_1(x - 1)}dx
\end{align*}
where
$$\zeta(t) = \frac{(1-x_1)A_1(t)}{x_1(1 - A_1(t))}.$$

If $\pi_{11} = \pi_{12} < \pi_{22}$, then it is easy to see from the stability analysis of (\ref{onefora}) that $$\lim_{t\to\infty}A_1(t) = 1\qquad\mbox{and, hence,}\qquad\lim_{t\to\infty}\zeta(t) = \infty.$$
Therefore, the mating pattern has the following formula:
\begin{equation}\label{rye}
Q_{12}(\infty)= x_1\theta_1\int_1^{\infty}x^{ - (\theta_1 +1)}e^{ - \left(\frac{x_1}{1 - x_1}\right)\theta_1(x - 1)}dx = \int_0^{\infty}\left(1+\frac{y}{x_1\theta_1}\right)^{ - \theta_1 -1}e^{ - \frac{y}{1 - x_1}}dy.
\end{equation}
Here, observe that $\theta_1>0$. 
Similarly, if $\pi_{11} = \pi_{12} > \pi_{22}$, then $$\lim_{t\to\infty}A_1(t) = 0\qquad\mbox{and, hence,}\qquad\lim_{t\to\infty}\zeta(t) = 0.$$
Therefore, the mating pattern has the following formula:
\begin{equation}\label{rye2} 
\begin{aligned}
	Q_{12}(\infty) &=   - x_1\theta_1 \int_0^1x^{ - (\theta_1 +1)}e^{ - \left(\frac{x_1}{1 - x_1}\right)\theta_1(x - 1)}dx \\&= \int_0^{-x_1\theta_1}\left(1+\frac{y}{x_1\theta_1}\right)^{ - \theta_1 -1}e^{ - \frac{y}{1 - x_1}}dy.
	\end{aligned}
\end{equation}
Here, observe that $\theta_1<0$.

For $\gamma=0$, that is, $\pi_{22}=\pi_{12}$, we relabel type-1 individuals as type-2 and type-2 individuals as type-1 (for each sex).  Hence, we have once again the situation where $\gamma=1$. Also, observe that $Q_{12}(t)=Q_{21}(t)$ since $X_i(t)=Y_i(t)$ for all $t\geq 0$. Hence, we get formulas for $Q_{12}(\infty)$ analogous to the ones in (\ref{rye}) and (\ref{rye2}) by simply swapping $\pi_{11}$ with $\pi_{22}$ and $x_1$ with $1-x_1$ (recall that $x_2=1-x_1$). More precisely, setting
\begin{equation*}
\theta_2 := \frac{\pi_{12}}{\pi_{11}-\pi_{12}},
\end{equation*}
we have
\begin{equation*}
Q_{12}(t)= \frac{\theta_2}{x_1}\int_{1-x_1}^{1-A_1(t)}\left(\frac{1 - x}{x_1}\right)^{\theta_2 - 1}\left(\frac{x}{1-x_1}\right)^{-\theta_2-1}\exp\left\{{-}\theta_2\left(\frac1{1-x} - \frac1{x_1}\right)\right\}dx.
\end{equation*} 
As before, by the stability analysis of $A_1(t)$, we have the following formulas for the mating pattern. If $\pi_{22} = \pi_{12} < \pi_{11}$, then
\begin{equation*}
\begin{aligned}
Q_{12}(\infty)&= (1-x_1)\theta_2\int_1^{\infty}x^{ - (\theta_2 +1)}e^{ - \left(\frac{1-x_1}{x_1}\right)\theta_2(x - 1)}dx 
\\&= \int_0^{\infty}\left(1+\frac{y}{(1-x_1)\theta_2}\right)^{ - \theta_2 -1}e^{ - \frac{y}{x_1}}dy,
\end{aligned}
\end{equation*}
where $\theta_2>0$. If $\pi_{22} = \pi_{12} > \pi_{11}$, then
\begin{equation*}
\begin{aligned}
Q_{12}(\infty) &=   -(1-x_1)\theta_2 \int_0^1x^{ - (\theta_2 +1)}e^{ - \left(\frac{1-x_1}{x_1}\right)\theta_2(x - 1)}dx 
\\&= \int_0^{-(1-x_1)\theta_2}\left(1+\frac{y}{(1-x_1)\theta_2}\right)^{ - \theta_2 -1}e^{ - \frac{y}{x_1}}dy,
\end{aligned}
\end{equation*}
where $\theta_2<0$.

\subsection{$\boldsymbol{\gamma\notin\{0,1\}}$}

$x_1=\gamma$ constitutes a special case and we study it first. 

\subsubsection{ $\boldsymbol{x_1 = \gamma\in(0,1)}$}

By (\ref{gama}) we have $\dot{A}_1 = 0$ and, therefore, $A_1(t) = x_1$. Plugging this in (\ref{syforz}) gives
$$\frac{\dot{Z}}{Z} =  - \pi_{12}x_1 - \pi_{22}(1 - x_1).$$
Using the initial condition $Z(0) = 1$, we get
$$Z(t) = e^{-(\pi_{12}x_1 + \pi_{22}(1-x_1))t}.$$
Finally,
$$\dot{Q}_{12} = \pi_{12}ZA_1(1 - A_1) = \pi_{12}x_1(1 - x_1)e^{-(\pi_{12}x_1 + \pi_{22}(1-x_1))t}$$
is easily solved with initial condition $Q_{12}(0) = 0$ to get
$$Q_{12}(t) = \frac{\pi_{12}x_1(1 - x_1)}{\pi_{12}x_1 + \pi_{22}(1-x_1)}\left(1 - e^{-(\pi_{12}x_1 + \pi_{22}(1-x_1))t}\right).$$
In particular, the mating pattern is given by
$$Q_{12}(\infty) = \frac{\pi_{12}x_1(1 - x_1)}{\pi_{12}x_1 + \pi_{22}(1-x_1)} = x_1(1-x_1)\Big[\frac{\pi_{12}\dt}{\pi_{12}(\pi_{22}-\pi_{12})+\pi_{22}(\pi_{11}-\pi_{12})}\Big].$$
Note that, by the definitions of $\theta_1$ and $\theta_2$, we have
\begin{equation*}
1+\frac{1}{\theta_1+\theta_2}=\frac{\pi_{12}(\pi_{22}-\pi_{12})+\pi_{22}(\pi_{11}-\pi_{12})}{\pi_{12}\dt}.
\end{equation*}
Hence, we can write
\be\label{basikelmi}
Q_{12}(\infty) = \frac{x_1(1-x_1)}{1+\frac1{\theta_1 + \theta_2}}.
\ee

\subsubsection{$\boldsymbol{x_1\neq\gamma}$}

Using partial fractions, (\ref{gama}) can be written as
\begin{equation}\label{neko}
\left( - \frac1{\gamma A_1} + \frac1{\gamma(1 - \gamma)(A_1 - \gamma)} + \frac1{(1-\gamma)(1 - A_1)}\right)\dot{A}_1 = -  (\pi_{11} + \pi_{22} - 2\pi_{12}).
\end{equation}
It is clear from (\ref{gama}) that $A_1(t)$ never crosses $\gamma$. Integrating both sides of (\ref{neko}) and using the initial condition $A_1(0) = x_1$, we get
$$\left(\frac{x_1(A_1(t) - \gamma)}{(x_1 - \gamma)A_1(t)}\right)^{\frac1{\gamma}}\left(\frac{(1-x_1)(A_1(t) - \gamma)}{(x_1 - \gamma)(1-A_1(t))}\right)^{\frac1{1-\gamma}} = e^{-(\pi_{11} + \pi_{22} - 2\pi_{12})t}.$$
Raising both sides to power $\frac{\pi_{12}}{\pi_{11} + \pi_{22} - 2\pi_{12}}$ gives
\begin{equation}\label{sakinol}
\left(\frac{x_1(A_1(t) - \gamma)}{(x_1 - \gamma)A_1(t)}\right)^{\theta_1}\left(\frac{(1-x_1)(A_1(t) - \gamma)}{(x_1 - \gamma)(1-A_1(t))}\right)^{\theta_2} = e^{-\pi_{12}t}.
\end{equation}
This is an implicit formula for $A_1(t)$. 

Next, we find a formula for $Z(t)$. We can rewrite (\ref{syforz}) as
\be\label{gams}
\frac{\dot Z}{Z}=-\dt A_1 (A_1-\gamma)-\pi_{12}A_1-\pi_{22}(1-A_1).
\ee
Note that (\ref{gama}) gives
\begin{align*}
&- (\pi_{11} + \pi_{22} - 2\pi_{12})A_1(A_1 - \gamma) = \frac{\dot{A}_1}{1-A_1},\\
&- A_1 = \frac{\dot{A}_1}{(\pi_{11} + \pi_{22} - 2\pi_{12})(A_1 - \gamma)(1-A_1)},\quad\mbox{and}\\
&- (1-A_1) = \frac{\dot{A}_1}{(\pi_{11} + \pi_{22} - 2\pi_{12})A_1(A_1 - \gamma)}.
\end{align*}
Substituting these into the right-hand side of (\ref{gams}) and using partial fractions, we get
\begin{align*}
\frac{\dot{Z}}{Z} &= \left(\frac1{1-A_1} + \frac{\pi_{12}}{(\pi_{11} + \pi_{22} - 2\pi_{12})(A_1 - \gamma)(1-A_1)} + \frac{\pi_{22}}{(\pi_{11} + \pi_{22} - 2\pi_{12})A_1(A_1 - \gamma)}\right)\dot{A}_1\\
& =\left( -\frac{\theta_1 +1}{A_1} + \frac{\theta_2 + 1}{1-A_1} + \frac{\theta_1 + \theta_2 + 1}{A_1 - \gamma}\right)\dot{A}_1.
\end{align*}
We integrate both sides, use the initial conditions $A_1(0) = x_1$ and $Z(0) = 1$, and (\ref{sakinol}) to deduce that
\begin{align}
Z(t) &= \left(\frac{A_1(t)}{x_1}\right)^{-\theta_1-1}\left(\frac{1 - A_1(t)}{1 - x_1}\right)^{-\theta_2-1}\left(\frac{A_1(t) - \gamma}{x_1 - \gamma}\right)^{\theta_1 + \theta_2 + 1}\label{degerli}\\
&= \left(\frac{x_1(1-x_1)(A_1(t) - \gamma)}{(x_1-\gamma)A_1(t)(1-A_1(t))}\right)\left(\frac{x_1(A_1(t) - \gamma)}{(x_1 - \gamma)A_1(t)}\right)^{\theta_1}\left(\frac{(1-x_1)(A_1(t) - \gamma)}{(x_1 - \gamma)(1-A_1(t))}\right)^{\theta_2}\nonumber\\
&= \left(\frac{x_1(1-x_1)(A_1(t) - \gamma)}{(x_1-\gamma)A_1(t)(1-A_1(t))}\right)e^{-\pi_{12}t}.\label{bgfg}
\end{align}
Here, the right-hand side of (\ref{degerli}) is in terms of $A_1(t)$ only. On the other hand, (\ref{bgfg}) is somewhat simpler.

Finally, we provide a formula for the limiting pair-type process. Note that (\ref{gama}) gives
$$A_1(1-A_1) = -\frac{\dot{A}_1}{(\pi_{11} + \pi_{22} - 2\pi_{12})(A_1 - \gamma)}.$$
Using this and (\ref{degerli}), we get
\begin{align*}
\dot{Q}_{12} &= \pi_{12}ZA_1(1 - A_1)\\
&= - \frac{\pi_{12}(x_1 - \gamma)^{-1}}{\pi_{11} + \pi_{22} - 2\pi_{12}}\left(\frac{A_1}{x_1}\right)^{-\theta_1-1}\left(\frac{1 - A_1}{1 - x_1}\right)^{-\theta_2-1}\left(\frac{A_1 - \gamma}{x_1 - \gamma}\right)^{\theta_1 + \theta_2}\dot A_1.
\end{align*}
Integrating both sides, using the initial conditions $A_1(0) = x_1$ and $Q_{12}(0) = 0$, and making a change of variables, we get
\begin{align*}
Q_{12}(t) &= - \frac{\pi_{12}(x_1 - \gamma)^{-1}}{\pi_{11} + \pi_{22} - 2\pi_{12}}\int_{x_1}^{A_1(t)}\left(\frac{x}{x_1}\right)^{-\theta_1-1}\left(\frac{1 - x}{1 - x_1}\right)^{-\theta_2-1}\left(\frac{x - \gamma}{x_1 - \gamma}\right)^{\theta_1 + \theta_2}dx\\
&=\frac{\pi_{12}}{\pi_{11} + \pi_{22} - 2\pi_{12}} \int_{0}^{\xi(t)}\left(1+\frac{\gamma y}{x_1}\right)^{-\theta_1-1}\left(1+\frac{(1-\gamma)y}{1-x_1}\right)^{-\theta_2 -1}dy
\end{align*}
where
$$\xi(t) = \frac{x_1 - A_1(t)}{A_1(t) - \gamma}.$$
In particular, the mating pattern is given by
\begin{equation}\label{sabahh}
 Q_{12}(\infty) = \frac{\pi_{12}}{\pi_{11} + \pi_{22} - 2\pi_{12}} \int_{0}^{\xi(\infty)}\left(1+\frac{\gamma y}{x_1}\right)^{-\theta_1-1}\left(1+\frac{(1-\gamma)y}{1-x_1}\right)^{-\theta_2 -1}dy.
\end{equation}
The value of $\xi(\infty)$ can be deduced from (\ref{gama}) using stability analysis:
\begin{itemize}

\item [(i)] If $\pi_{11}>\pi_{12}$ and $\pi_{22}>\pi_{12}$, then $0<\gamma<1$, $A_1(\infty) = \gamma$ and $\xi(\infty) = \infty$.
\item [(ii)] If $\pi_{11}<\pi_{12}$ and $\pi_{22}<\pi_{12}$, then $0<\gamma<1$ and there are two subcases.
\begin{itemize}
\item If $x_1<\gamma$, then $A_1(\infty) = 0$ and $\xi(\infty) =  - x_1/\gamma$.
\item If $x_1>\gamma$, then $A_1(\infty) = 1$ and $\xi(\infty) = - (1- x_1)/(1-\gamma)$.
\end{itemize}
\item [(iii)] If $\pi_{11}>\pi_{12}$ and $\pi_{22}<\pi_{12}$, then there are two subcases.
\begin{itemize}
\item If $\pi_{11} + \pi_{22} < 2\pi_{12}$, then $\gamma>1$, $A_1(\infty) = 0$ and $\xi(\infty) =  - x_1/\gamma$.
\item If $\pi_{11} + \pi_{22} > 2\pi_{12}$, then $\gamma<0$, $A_1(\infty) = 0$ and $\xi(\infty) =  - x_1/\gamma$.
\end{itemize}
\item [(iv)] If $\pi_{11}<\pi_{12}$ and $\pi_{22}>\pi_{12}$, then there are two subcases.
\begin{itemize}
\item If $\pi_{11} + \pi_{22} < 2\pi_{12}$, then $\gamma<0$, $A_1(\infty) = 1$ and $\xi(\infty) = -(1- x_1)/(1-\gamma)$.
\item If $\pi_{11} + \pi_{22} > 2\pi_{12}$, then $\gamma>1$, $A_1(\infty) = 1$ and $\xi(\infty) = -(1- x_1)/(1-\gamma)$.
\end{itemize}
\end{itemize}
Hence, we have an explicit formula for the mating pattern in each case. 
\begin{figure}[H]
    \centering
    \includegraphics[trim= 15mm 52mm 00mm 45mm, clip, scale=0.4]{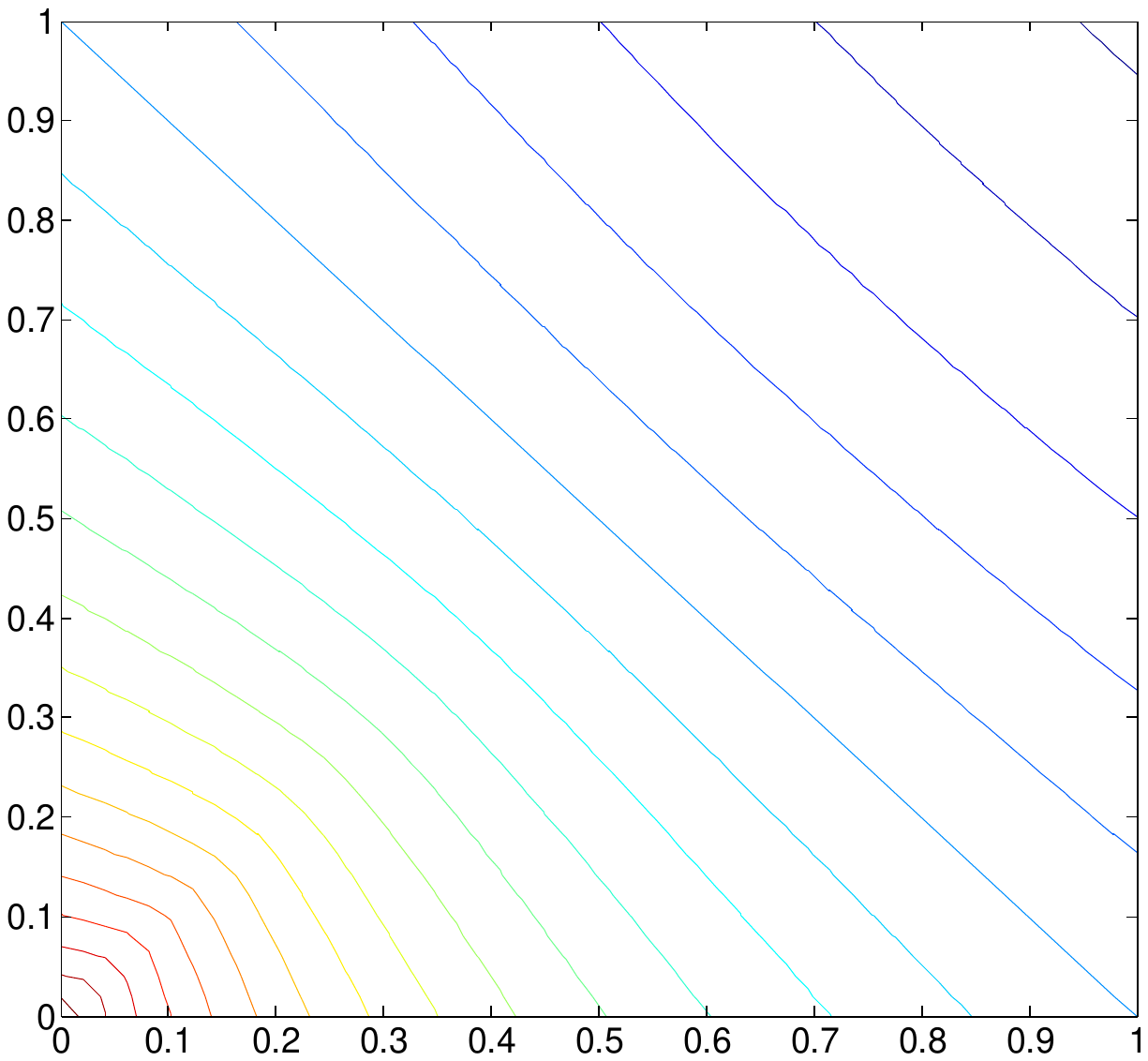}
    \caption{Level curves of $Q_{12}(\infty)$ as a function of $\pi_{11}$ ($x$-axis) and $\pi_{22}$ ($y$-axis) for fixed $\pi_{12}=\pi_{21}=1/2$ and $x_1=x_2=y_1=y_2=1/2$. The value of $Q_{12}(\infty)$ on each level curve is separated by 1/64. The diagonal line $\pi_{11}+\pi_{22}=1$ corresponds to panmixia on which $Q_{12}(\infty) = x_1y_2 = 1/4$.}
    \label{levelcurves}
\end{figure}

\subsection{Characterization of homogamy/panmixia/heterogamy}

Having derived an explicit formula for the mating pattern in the symmetric $2\times 2$ case, we use this formula to provide a trichotomy regarding the mating preferences vs.\ the mating pattern.
\begin{theorem}\label{homhet}
Assume $k=2$, $\pi_{12} = \pi_{21}$ and $x_1 = y_1 \in (0,1)$. Then, the following hold.
\begin{itemize}
\item [(a)] $\textup{(homogamy for symmetric pop.)}\quad \ Q_{12}(\infty) < x_1(1-x_1)\quad\text{if}\quad\pi_{11} + \pi_{22} > 2\pi_{12}$.
\item [(b)] $\textup{(panmixia for symmetric pop.)}\quad\ \ \ Q_{12}(\infty) = x_1(1-x_1)\quad\text{if}\quad\pi_{11} + \pi_{22} = 2\pi_{12}$. 
\item [(c)] $\textup{(heterogamy for symmetric pop.)}\quad Q_{12}(\infty) > x_1(1-x_1)\quad\text{if}\quad\pi_{11} + \pi_{22} < 2\pi_{12}$.
\end{itemize}
\end{theorem}

\begin{remark}
The analog of the trichotomy in Theorem \ref{homhet} for finite populations (without imposing any symmetry conditions) was established in \cite[Theorem 3.9]{GunYil14a} and recorded in \eqref{ortakavm}. In fact, Theorem \ref{homhet} can be almost obtained from \eqref{ortakavm} by applying Theorem \ref{cormat} and the dominated convergence theorem, except that the strict inequalities would not necessarily be preserved. Our main motivation for including Theorem \ref{homhet} here is to provide an application of our formula for the mating pattern.
\end{remark}

\begin{proof}[Proof of Theorem \ref{homhet}(a)]
 
We analyze the formula we derived for $Q_{12}(\infty)$ which depends on $\gamma$.

\underline{$\gamma\in\{0,1\}$:} Consider the case $\gamma=1$, that is, $\pi_{11}=\pi_{12}$. Since we assume that $\pi_{11} + \pi_{22} > 2\pi_{12}$, we have $\pi_{11}=\pi_{12}<\pi_{22}$. Then (\ref{rye}) holds for the mating pattern:
$$Q_{12}(\infty)= \int_0^{\infty}\left(1+\frac{y}{x_1\theta_1}\right)^{ - \theta_1 -1}e^{ - \frac{y}{1 - x_1}}dy,$$
with $\theta_1 > 0$. Note that, since $e^y>(1+y/c)^c > 0$ for every $y>0$ and $c>0$, we have
$$e^{-\frac{y}{1-x_1}} < \left(1+\frac{y}{x_1\theta_1}\right)^{-\frac{x_1}{1-x_1}\theta_1}.$$
Therefore,
$$Q_{12}(\infty) < \int_0^{\infty}\left(1+\frac{y}{x_1\theta_1}\right)^{ - \theta_1 - 1 - \frac{x_1}{1-x_1}\theta_1}dy = x_1(1-x_1).$$
The proof for $\gamma=0$, that is, $\pi_{22}=\pi_{12}$, is exactly the same.

\underline{$\gamma\notin\{0,1\}$:}
In the case where $x_1=\gamma\in(0,1)$, recall from \eqref{basikelmi} that
\begin{equation*}
Q_{12}(\infty)=\frac{x_1(1-x_1)}{1+\frac1{\theta_1 + \theta_2}}.
\end{equation*}
Since $\gamma\in(0,1)$ and $\pi_{11} + \pi_{22} > 2\pi_{12}$, we have $\pi_{11}>\pi_{12}$ and $\pi_{22}>\pi_{12}$. Hence, $\theta_1+\theta_2>0$, which implies that $Q_{12}(\infty)<x_1(1-x_1)$.

Now assume that $x_1\not= \gamma$. We consider first the case $\pi_{11}>\pi_{12}$ and $\pi_{22}>\pi_{12}$. Then, $\gamma\in(0,1)$, $A_1(\infty)=\gamma$, $\xi(\infty)=\infty$, and $\theta_1,\theta_2>0$. By the formula in (\ref{sabahh}) we have
\begin{equation*}
Q_{12}(\infty) = \theta_1\gamma \int_{0}^{\infty}\left(1+\frac{\gamma y}{x_1}\right)^{-(\theta_1+1)}\left(1+\frac{(1-\gamma)y}{1-x_1}\right)^{-(\theta_2 +1)}dy.
\end{equation*}
If $x_1>\gamma$, then
$$0 < \frac{1-x_1}{1-\gamma} < 1 < \frac{x_1}{\gamma}.$$
Thus, since $(1+y/c_1)^{c_1}<(1+y/c_2)^{c_2}$ for every $y>0$ and $0<c_1<c_2$, we get
\begin{equation*}
0 < \left(1+\frac{(1-\gamma)y}{1-x_1}\right)^{\frac{1-x_1}{1-\gamma}} < \left(1+\frac{\gamma y}{x_1}\right)^{\frac{x_1}{\gamma}}.
\end{equation*}
The above inequality gives
$$\left(1+\frac{\gamma y}{x_1}\right)^{-(\theta_1 + 1)} < \left(1+\frac{(1-\gamma)y}{1-x_1}\right)^{-(\theta_1 + 1)\frac{\gamma(1-x_1)}{(1-\gamma)x_1}}.$$
Therefore, 
\begin{align*}
Q_{12}(\infty) &< \theta_1\gamma \int_{0}^{\infty}\left(1+\frac{(1-\gamma)y}{1-x_1}\right)^{-(\theta_1 + 1)\frac{\gamma(1-x_1)}{(1-\gamma)x_1} - (\theta_2 +1)}dy = \frac{x_1(1-x_1)}{1 + (1-x_1)/\theta_1} < x_1(1-x_1).
\end{align*}
Similarly, if $x_1<\gamma$, then we get
$$Q_{12}(\infty) < \frac{x_1(1-x_1)}{1 + x_1/\theta_2} < x_1(1-x_1).$$

Next, consider the case $\pi_{11}>\pi_{12}>\pi_{22}$. Then $\gamma<0$, $A_1(\infty)=0$, $\xi(\infty)=-x_1/\gamma$, $\theta_1<-1$, $\theta_2>0$ and $\theta_1+\theta_2<0$. By (\ref{sabahh}), we have 
\begin{equation*}
Q_{12}(\infty) = \theta_1\gamma \int_0^{-x_1/\gamma}\left(1+\frac{\gamma y}{x_1}\right)^{-(\theta_1+1)}\left(1+\frac{(1-\gamma)y}{1-x_1}\right)^{-(\theta_2 +1)}dy.
\end{equation*}
Since $\gamma<0$,
$$\frac{1-x_1}{1-\gamma} > 0 > \frac{x_1}{\gamma},$$
which implies
$$\left(1+\frac{\gamma y}{x_1}\right)^{\frac{x_1}{\gamma}} >  \left(1+\frac{(1-\gamma)y}{1-x_1}\right)^{\frac{1-x_1}{1-\gamma}} > 0$$ 
for every $y\in (0,-x_1/\gamma)$. Note that $(1-\gamma)\theta_2=\gamma \theta_1$. Hence, raising both sides of the above inequality to power $-\frac{\gamma\theta_1}{\gamma-x_1}=-\frac{(1-\gamma)\theta_2}{\gamma-x_1}>0$, we get
$$\left(1+\frac{\gamma y}{x_1}\right)^{-\frac{x_1\theta_1}{\gamma-x_1}} > \left(1+\frac{(1-\gamma)y}{1-x_1}\right)^{-\frac{(1-x_1)\theta_2}{\gamma-x_1}}.$$
Therefore,
\begin{align*}
Q_{12}(\infty) &< \theta_1\gamma \int_0^{-x_1/\gamma}\left(1+\frac{\gamma y}{x_1}\right)^{-(\theta_1+1) -\frac{x_1\theta_1}{\gamma-x_1}}\left(1+\frac{(1-\gamma)y}{1-x_1}\right)^{-(\theta_2 +1) + \frac{(1-x_1)\theta_2}{\gamma-x_1}}dy\\
&= \theta_1\gamma \int_0^{-x_1/\gamma}\left(\frac{1+\frac{\gamma y}{x_1}}{1+\frac{(1-\gamma)y}{1-x_1}}\right)^{-\frac{\theta_1\gamma}{\gamma - x_1} - 1}\left(1+\frac{(1-\gamma)y}{1-x_1}\right)^{-2}dy\\
&= x_1(1-x_1)\left(\frac{\theta_1\gamma}{x_1-\gamma}\right) \int_0^1u^{\frac{\theta_1\gamma}{x_1 - \gamma} - 1}du\\
&= x_1(1-x_1).
\end{align*}

Finally, the case $\pi_{22}>\pi_{12}>\pi_{11}$ is reduced to the previous case simply by switching the roles of $\pi_{11}$ and $\pi_{22}$ (and of $x_1$ and $1-x_1$).
\end{proof}

\begin{proof}[Proof of Theorem \ref{homhet}(b)]
If $\pi_{11} + \pi_{22} = 2\pi_{12}$, then the fine balance condition is satisfied, and $Q_{12}(\infty)=x_1 y_2 = x_1(1-x_1)$ by Theorem \ref{thmfb} and the assumption that $x_1 = y_1$.
\end{proof}

\begin{proof}[Proof of Theorem \ref{homhet}(c)]
	
We proceed exactly as in the proof of Theorem \ref{homhet}(a). 
	
	\underline{$\gamma\in\{0,1\}$:} Consider the case $\gamma=1$, that is, $\pi_{11}=\pi_{12}$. Since we assume that $\pi_{11} + \pi_{22} < 2\pi_{12}$, we have $\pi_{11}=\pi_{12}>\pi_{22}$. Then (\ref{rye2}) holds for the mating pattern:
	$$Q_{12}(\infty)= \int_0^{-x_1\theta_1}\left(1+\frac{y}{x_1\theta_1}\right)^{ - \theta_1 -1}e^{ - \frac{y}{1 - x_1}}dy,$$
	with $\theta_1 < 0$. Note that, since $e^y<(1+y/c)^c$ for every $y\in(0,-c)$ and $c<0$, we have
	$$e^{-\frac{y}{1-x_1}} > \left(1+\frac{y}{x_1\theta_1}\right)^{-\frac{x_1}{1-x_1}\theta_1} > 0$$
	for every $y\in(0,-x_1\theta_1)$.
	Therefore,
	$$Q_{12}(\infty) > \int_0^{-x_1\theta_1}\left(1+\frac{y}{x_1\theta_1}\right)^{ - \theta_1 - 1 - \frac{x_1}{1-x_1}\theta_1}dy = x_1(1-x_1).$$
	The proof for $\gamma=0$, that is, $\pi_{22}=\pi_{12}$, is exactly the same.
	
	\underline{$\gamma\notin\{0,1\}$:}
	In the case where $x_1=\gamma\in(0,1)$, recall from \eqref{basikelmi} that
	\begin{equation*}
	Q_{12}(\infty)=\frac{x_1(1-x_1)}{1+\frac1{\theta_1 + \theta_2}}.
	\end{equation*}
	Since $\gamma\in(0,1)$ and $\pi_{11} + \pi_{22} < 2\pi_{12}$, we have $\pi_{11}<\pi_{12}$ and $\pi_{22}<\pi_{12}$. Hence, $\theta_1+\theta_2<-2$, which implies that $Q_{12}(\infty)>x_1(1-x_1)$.
	
	Now assume that $x_1\not= \gamma$. We consider first the case $\pi_{11}<\pi_{12}$ and $\pi_{22}<\pi_{12}$. Then, $\gamma\in(0,1)$, and $\theta_1,\theta_2 < -1$. By the formula in (\ref{sabahh}) we have
	\begin{equation*}
	Q_{12}(\infty) = \theta_1\gamma \int_{0}^{\xi(\infty)}\left(1+\frac{\gamma y}{x_1}\right)^{-(\theta_1+1)}\left(1+\frac{(1-\gamma)y}{1-x_1}\right)^{-(\theta_2 +1)}dy.
	\end{equation*}
	If $x_1>\gamma$, then $A_1(\infty)=1$, $\xi(\infty)= -(1-x_1)/(1-\gamma)$ and
	$$0 < \frac{1-x_1}{1-\gamma} < 1 < \frac{x_1}{\gamma}.$$
	Thus, since $(1+y/c_1)^{c_1}<(1+y/c_2)^{c_2}$ for every $y\in(-c_1,0)$ and $0< c_1<c_2$, we get
	\begin{equation*}
	0 < \left(1+\frac{(1-\gamma)y}{1-x_1}\right)^{\frac{1-x_1}{1-\gamma}} < \left(1+\frac{\gamma y}{x_1}\right)^{\frac{x_1}{\gamma}}
	\end{equation*}
	for every $y\in(-(1-x_1)/(1-\gamma),0)$.
	The above inequality gives
	$$\left(1+\frac{\gamma y}{x_1}\right)^{-(\theta_1 + 1)} > \left(1+\frac{(1-\gamma)y}{1-x_1}\right)^{-(\theta_1 + 1)\frac{\gamma(1-x_1)}{(1-\gamma)x_1}}.$$
	Therefore, 
	\begin{align*}
	Q_{12}(\infty) &> -\theta_1\gamma \int_{-(1-x_1)/(1-\gamma)}^0\left(1+\frac{(1-\gamma)y}{1-x_1}\right)^{-(\theta_1 + 1)\frac{\gamma(1-x_1)}{(1-\gamma)x_1} - (\theta_2 +1)}dy\\
	&= \frac{x_1(1-x_1)}{1 + (1-x_1)/\theta_1} > x_1(1-x_1).
	\end{align*}
	Similarly, if $x_1<\gamma$, then we get
	$$Q_{12}(\infty) > \frac{x_1(1-x_1)}{1 + x_1/\theta_2} > x_1(1-x_1).$$
	
	Next, consider the case $\pi_{11}>\pi_{12}>\pi_{22}$. Then $\gamma>1$, $A_1(\infty)=0$, $\xi(\infty)=-x_1/\gamma$, $\theta_1<-1$, $\theta_2>0$ and $\theta_1+\theta_2>0$. By (\ref{sabahh}), we have 
	\begin{equation*}
	Q_{12}(\infty) = -\theta_1\gamma \int_{-x_1/\gamma}^0\left(1+\frac{\gamma y}{x_1}\right)^{-(\theta_1+1)}\left(1+\frac{(1-\gamma)y}{1-x_1}\right)^{-(\theta_2 +1)}dy.
	\end{equation*}
	Since $\gamma>1$,
	$$\frac{1-x_1}{1-\gamma} < 0 < \frac{x_1}{\gamma},$$
	which implies
	$$0 < \left(1+\frac{\gamma y}{x_1}\right)^{\frac{x_1}{\gamma}} < \left(1+\frac{(1-\gamma)y}{1-x_1}\right)^{\frac{1-x_1}{1-\gamma}}$$ 
	for every $y\in (-x_1/\gamma,0)$. Note that $(1-\gamma)\theta_2=\gamma \theta_1$. Hence, raising both sides of the above inequality to power $-\frac{\gamma\theta_1}{\gamma-x_1}=-\frac{(1-\gamma)\theta_2}{\gamma-x_1}>0$, we get
	$$\left(1+\frac{\gamma y}{x_1}\right)^{-\frac{x_1\theta_1}{\gamma-x_1}} < \left(1+\frac{(1-\gamma)y}{1-x_1}\right)^{-\frac{(1-x_1)\theta_2}{\gamma-x_1}}.$$
	Therefore,
	\begin{align*}
	Q_{12}(\infty) &> -\theta_1\gamma \int_{-x_1/\gamma}^0\left(1+\frac{\gamma y}{x_1}\right)^{-(\theta_1+1) -\frac{x_1\theta_1}{\gamma-x_1}}\left(1+\frac{(1-\gamma)y}{1-x_1}\right)^{-(\theta_2 +1) + \frac{(1-x_1)\theta_2}{\gamma-x_1}}dy\\
	&= -\theta_1\gamma \int_{-x_1/\gamma}^0\left(\frac{1+\frac{\gamma y}{x_1}}{1+\frac{(1-\gamma)y}{1-x_1}}\right)^{-\frac{\theta_1\gamma}{\gamma - x_1} - 1}\left(1+\frac{(1-\gamma)y}{1-x_1}\right)^{-2}dy\\
	&= x_1(1-x_1)\left(\frac{\theta_1\gamma}{x_1-\gamma}\right) \int_0^1u^{\frac{\theta_1\gamma}{x_1 - \gamma} - 1}du\\
	&= x_1(1-x_1).
	\end{align*}
	
	Finally, the case $\pi_{22}>\pi_{12}>\pi_{11}$ is reduced to the previous case simply by switching the roles of $\pi_{11}$ and $\pi_{22}$ (and of $x_1$ and $1-x_1$).
\end{proof}

\section*{Acknowledgments}

We thank A.\ Courtiol, R.\ O'Donnell and F.\ Rezakhanlou for valuable discussions. O.\ G\"un gratefully acknowledges support by DFG SPP Priority Programme 1590 ``Probabilistic Structures in Evolution". A.\ Yilmaz is supported in part by European Union FP7 Marie Curie Career Integration Grant no.\ 322078.

\bibliographystyle{plain}
\bibliography{mating_references}

\end{document}